\documentclass[12pt,reqno]{amsart}       
\usepackage{txfonts}
\usepackage{amssymb}
\usepackage{eucal}
\usepackage{graphicx}
\usepackage{amsmath}
\usepackage{amscd}
\usepackage[all]{xy}           
\usepackage{tikz}
\usepackage{amsfonts,latexsym}
\usepackage{xspace}
\usepackage{epsfig}
\usepackage{float}
\usepackage{color}
\usepackage{fancybox}
\usepackage{colordvi}
\usepackage{multicol}
\usepackage{colordvi}
\usepackage[colorlinks,final,backref=page,hyperindex,hypertex]{hyperref}
\usepackage[active]{srcltx} 

\newcommand {\emptycomment}[1]{} 

\newcommand{\nc}{\newcommand}
\newcommand{\delete}[1]{}
\nc{\mfootnote}[1]{\footnote{#1}} 
\nc{\todo}[1]{\tred{To do:} #1}

\nc{\mlabel}[1]{\label{#1}}  
\nc{\mcite}[1]{\cite{#1}}  
\nc{\mref}[1]{\ref{#1}}  
\nc{\mbibitem}[1]{\bibitem{#1}} 

\delete{
\nc{\mlabel}[1]{\label{#1}  
{\hfill \hspace{1cm}{\bf{{\ }\hfill(#1)}}}}
\nc{\mcite}[1]{\cite{#1}{{\bf{{\ }(#1)}}}}  
\nc{\mref}[1]{\ref{#1}{{\bf{{\ }(#1)}}}}  
\nc{\mbibitem}[1]{\bibitem[\bf #1]{#1}} 
}

\newtheorem{thm}{Theorem}[section]
\newtheorem{lem}[thm]{Lemma}
\newtheorem{cor}[thm]{Corollary}
\newtheorem{pro}[thm]{Proposition}
\newtheorem{ex}[thm]{Example}
\newtheorem{rmk}[thm]{Remark}
\newtheorem{defi}[thm]{Definition}

\nc{\tred}[1]{\textcolor{red}{#1}}
\nc{\tblue}[1]{\textcolor{blue}{#1}}
\nc{\tgreen}[1]{\textcolor{green}{#1}}
\nc{\tpurple}[1]{\textcolor{purple}{#1}}
\nc{\btred}[1]{\textcolor{red}{\bf #1}}
\nc{\btblue}[1]{\textcolor{blue}{\bf #1}}
\nc{\btgreen}[1]{\textcolor{green}{\bf #1}}
\nc{\btpurple}[1]{\textcolor{purple}{\bf #1}}

\nc{\dcy}[1]{\textcolor{purple}{Chengyu:#1}}
\nc{\cm}[1]{\textcolor{red}{Chengming:#1}}
\nc{\li}[1]{\textcolor{blue}{#1}}
\nc{\lir}[1]{\textcolor{blue}{Li: #1}}
\nc{\lit}[2]{\textcolor{blue}{#1}{}} 
\nc{\yh}[1]{\textcolor{green}{Yunhe: #1}}

\nc{\twovec}[2]{\left(\begin{array}{c} #1 \\ #2\end{array} \right )}
\nc{\threevec}[3]{\left(\begin{array}{c} #1 \\ #2 \\ #3 \end{array}\right )}
\nc{\twomatrix}[4]{\left(\begin{array}{cc} #1 & #2\\ #3 & #4 \end{array} \right)}
\nc{\threematrix}[9]{{\left(\begin{matrix} #1 & #2 & #3\\ #4 & #5 & #6 \\ #7 & #8 & #9 \end{matrix} \right)}}
\nc{\twodet}[4]{\left|\begin{array}{cc} #1 & #2\\ #3 & #4 \end{array} \right|}

\nc{\rk}{\mathrm{r}}

\nc{\gensp}{V} 
\nc{\relsp}{\Lambda} 
\nc{\leafsp}{X}    
\nc{\treesp}{\overline{\calt}} 

\nc{\vin}{{\mathrm Vin}}    
\nc{\lin}{{\mathrm Lin}}    

\nc{\gop}{{\,\omega\,}}     
\nc{\gopb}{{\,\nu\,}}
\nc{\svec}[2]{{\tiny\left(\begin{matrix}#1\\
#2\end{matrix}\right)\,}}  
\nc{\ssvec}[2]{{\tiny\left(\begin{matrix}#1\\
#2\end{matrix}\right)\,}} 

\nc{\typeI}{local cocycle $3$-Lie bialgebra\xspace}
\nc{\typeIs}{local cocycle $3$-Lie bialgebras\xspace}
\nc{\typeII}{double construction $3$-Lie bialgebra\xspace}
\nc{\typeIIs}{double construction $3$-Lie bialgebras\xspace}

\nc{\bia}{{$\mathcal{P}$-bimodule ${\bf k}$-algebra}\xspace}
\nc{\bias}{{$\mathcal{P}$-bimodule ${\bf k}$-algebras}\xspace}

\nc{\rmi}{{\mathrm{I}}}
\nc{\rmii}{{\mathrm{II}}}
\nc{\rmiii}{{\mathrm{III}}}
\nc{\pr}{{\mathrm{pr}}}

\nc{\ttau}{\tau}
\nc{\tS}{S}

\nc{\pll}{\beta}
\nc{\plc}{\epsilon}

\nc{\ass}{{\mathit{Ass}}}
\nc{\lie}{{\mathit{Lie}}}
\nc{\comm}{{\mathit{Comm}}}
\nc{\dend}{{\mathit{Dend}}}
\nc{\zinb}{{\mathit{Zinb}}}
\nc{\tdend}{{\mathit{TDend}}}
\nc{\prelie}{{\mathit{preLie}}}
\nc{\postlie}{{\mathit{PostLie}}}
\nc{\quado}{{\mathit{Quad}}}
\nc{\octo}{{\mathit{Octo}}}
\nc{\ldend}{{\mathit{ldend}}}
\nc{\lquad}{{\mathit{LQuad}}}

 \nc{\adec}{\check{;}} \nc{\aop}{\alpha}
\nc{\dftimes}{\widetilde{\otimes}} \nc{\dfl}{\succ} \nc{\dfr}{\prec}
\nc{\dfc}{\circ} \nc{\dfb}{\bullet} \nc{\dft}{\star}
\nc{\dfcf}{{\mathbf k}} \nc{\apr}{\ast} \nc{\spr}{\cdot}
\nc{\twopr}{\circ} \nc{\tspr}{\star} \nc{\sempr}{\ast}
\nc{\disp}[1]{\displaystyle{#1}}
\nc{\bin}[2]{ (_{\stackrel{\scs{#1}}{\scs{#2}}})}  
\nc{\binc}[2]{ \left (\!\! \begin{array}{c} \scs{#1}\\
    \scs{#2} \end{array}\!\! \right )}  
\nc{\bincc}[2]{  \left ( {\scs{#1} \atop
    \vspace{-.5cm}\scs{#2}} \right )}  
\nc{\sarray}[2]{\begin{array}{c}#1 \vspace{.1cm}\\ \hline
    \vspace{-.35cm} \\ #2 \end{array}}
\nc{\bs}{\bar{S}} \nc{\dcup}{\stackrel{\bullet}{\cup}}
\nc{\dbigcup}{\stackrel{\bullet}{\bigcup}} \nc{\etree}{\big |}
\nc{\la}{\longrightarrow} \nc{\fe}{\'{e}} \nc{\rar}{\rightarrow}
\nc{\dar}{\downarrow} \nc{\dap}[1]{\downarrow
\rlap{$\scriptstyle{#1}$}} \nc{\uap}[1]{\uparrow
\rlap{$\scriptstyle{#1}$}} \nc{\defeq}{\stackrel{\rm def}{=}}
\nc{\dis}[1]{\displaystyle{#1}} \nc{\dotcup}{\,
\displaystyle{\bigcup^\bullet}\ } \nc{\sdotcup}{\tiny{
\displaystyle{\bigcup^\bullet}\ }} \nc{\hcm}{\ \hat{,}\ }
\nc{\hcirc}{\hat{\circ}} \nc{\hts}{\hat{\shpr}}
\nc{\lts}{\stackrel{\leftarrow}{\shpr}}
\nc{\rts}{\stackrel{\rightarrow}{\shpr}} \nc{\lleft}{[}
\nc{\lright}{]} \nc{\uni}[1]{\tilde{#1}} \nc{\wor}[1]{\check{#1}}
\nc{\free}[1]{\bar{#1}} \nc{\den}[1]{\check{#1}} \nc{\lrpa}{\wr}
\nc{\curlyl}{\left \{ \begin{array}{c} {} \\ {} \end{array}
    \right .  \!\!\!\!\!\!\!}
\nc{\curlyr}{ \!\!\!\!\!\!\!
    \left . \begin{array}{c} {} \\ {} \end{array}
    \right \} }
\nc{\leaf}{\ell}       
\nc{\longmid}{\left | \begin{array}{c} {} \\ {} \end{array}
    \right . \!\!\!\!\!\!\!}
\nc{\ot}{\otimes} \nc{\sot}{{\scriptstyle{\ot}}}
\nc{\otm}{\overline{\ot}}
\nc{\ora}[1]{\stackrel{#1}{\rar}}
\nc{\ola}[1]{\stackrel{#1}{\la}}
\nc{\pltree}{\calt^\pl}
\nc{\epltree}{\calt^{\pl,\NC}}
\nc{\rbpltree}{\calt^r}
\nc{\scs}[1]{\scriptstyle{#1}} \nc{\mrm}[1]{{\rm #1}}
\nc{\dirlim}{\displaystyle{\lim_{\longrightarrow}}\,}
\nc{\invlim}{\displaystyle{\lim_{\longleftarrow}}\,}
\nc{\mvp}{\vspace{0.5cm}} \nc{\svp}{\vspace{2cm}}
\nc{\vp}{\vspace{8cm}} \nc{\proofbegin}{\noindent{\bf Proof: }}
\nc{\proofend}{$\blacksquare$ \vspace{0.5cm}}
\nc{\freerbpl}{{F^{\mathrm RBPL}}}
\nc{\sha}{{\mbox{\cyr X}}}  
\nc{\ncsha}{{\mbox{\cyr X}^{\mathrm NC}}} \nc{\ncshao}{{\mbox{\cyr
X}^{\mathrm NC,\,0}}}
\nc{\shpr}{\diamond}    
\nc{\shprm}{\overline{\diamond}}    
\nc{\shpro}{\diamond^0}    
\nc{\shprr}{\diamond^r}     
\nc{\shpra}{\overline{\diamond}^r}
\nc{\shpru}{\check{\diamond}} \nc{\catpr}{\diamond_l}
\nc{\rcatpr}{\diamond_r} \nc{\lapr}{\diamond_a}
\nc{\sqcupm}{\ot}
\nc{\lepr}{\diamond_e} \nc{\vep}{\varepsilon} \nc{\labs}{\mid\!}
\nc{\rabs}{\!\mid} \nc{\hsha}{\widehat{\sha}}
\nc{\lsha}{\stackrel{\leftarrow}{\sha}}
\nc{\rsha}{\stackrel{\rightarrow}{\sha}} \nc{\lc}{\lfloor}
\nc{\rc}{\rfloor}
\nc{\tpr}{\sqcup}
\nc{\nctpr}{\vee}
\nc{\plpr}{\star}
\nc{\rbplpr}{\bar{\plpr}}
\nc{\sqmon}[1]{\langle #1\rangle}
\nc{\forest}{\calf}
\nc{\altx}{\Lambda_X} \nc{\vecT}{\vec{T}} \nc{\onetree}{\bullet}
\nc{\Ao}{\check{A}}
\nc{\seta}{\underline{\Ao}}
\nc{\deltaa}{\overline{\delta}}
\nc{\trho}{\tilde{\rho}}

\nc{\rpr}{\circ}
\nc{\dpr}{{\tiny\diamond}}
\nc{\rprpm}{{\rpr}}

\nc{\mmbox}[1]{\mbox{\ #1\ }} \nc{\ann}{\mrm{ann}}
\nc{\Aut}{\mrm{Aut}} \nc{\can}{\mrm{can}}
\nc{\twoalg}{{two-sided algebra}\xspace}
\nc{\colim}{\mrm{colim}}
\nc{\Cont}{\mrm{Cont}} \nc{\rchar}{\mrm{char}}
\nc{\cok}{\mrm{coker}} \nc{\dtf}{{R-{\rm tf}}} \nc{\dtor}{{R-{\rm
tor}}}

\nc{\depth}{{\mrm d}}
\nc{\Div}{{\mrm Div}} \nc{\End}{\mrm{End}} \nc{\Ext}{\mrm{Ext}}
\nc{\Fil}{\mrm{Fil}} \nc{\Frob}{\mrm{Frob}} \nc{\Gal}{\mrm{Gal}}
\nc{\GL}{\mrm{GL}} \nc{\Hom}{\mrm{Hom}} \nc{\hsr}{\mrm{H}}
\nc{\hpol}{\mrm{HP}} \nc{\id}{\mrm{id}} \nc{\im}{\mrm{im}}
\nc{\incl}{\mrm{incl}} \nc{\length}{\mrm{length}}
\nc{\LR}{\mrm{LR}} \nc{\mchar}{\rm char} \nc{\NC}{\mrm{NC}}
\nc{\mpart}{\mrm{part}} \nc{\pl}{\mrm{PL}}
\nc{\ql}{{\QQ_\ell}} \nc{\qp}{{\QQ_p}}
\nc{\rank}{\mrm{rank}} \nc{\rba}{\rm{RBA }} \nc{\rbas}{\rm{RBAs }}
\nc{\rbpl}{\mrm{RBPL}}
\nc{\rbw}{\rm{RBW }} \nc{\rbws}{\rm{RBWs }} \nc{\rcot}{\mrm{cot}}
\nc{\rest}{\rm{controlled}\xspace}
\nc{\rdef}{\mrm{def}} \nc{\rdiv}{{\rm div}} \nc{\rtf}{{\rm tf}}
\nc{\rtor}{{\rm tor}} \nc{\res}{\mrm{res}} \nc{\SL}{\mrm{SL}}
\nc{\Spec}{\mrm{Spec}} \nc{\tor}{\mrm{tor}} \nc{\Tr}{\mrm{Tr}}
\nc{\mtr}{\mrm{sk}}

\nc{\ab}{\mathbf{Ab}} \nc{\Alg}{\mathbf{Alg}}
\nc{\Algo}{\mathbf{Alg}^0} \nc{\Bax}{\mathbf{Bax}}
\nc{\Baxo}{\mathbf{Bax}^0} \nc{\RB}{\mathbf{RB}}
\nc{\RBo}{\mathbf{RB}^0} \nc{\BRB}{\mathbf{RB}}
\nc{\Dend}{\mathbf{DD}} \nc{\bfk}{{\bf k}} \nc{\bfone}{{\bf 1}}
\nc{\base}[1]{{a_{#1}}} \nc{\detail}{\marginpar{\bf More detail}
    \noindent{\bf Need more detail!}
    \svp}
\nc{\Diff}{\mathbf{Diff}} \nc{\gap}{\marginpar{\bf
Incomplete}\noindent{\bf Incomplete!!}
    \svp}
\nc{\FMod}{\mathbf{FMod}} \nc{\mset}{\mathbf{MSet}}
\nc{\rb}{\mathrm{RB}} \nc{\Int}{\mathbf{Int}}
\nc{\Mon}{\mathbf{Mon}}
\nc{\remarks}{\noindent{\bf Remarks: }}
\nc{\OS}{\mathbf{OS}} 
\nc{\Rep}{\mathbf{Rep}}
\nc{\Rings}{\mathbf{Rings}} \nc{\Sets}{\mathbf{Sets}}
\nc{\DT}{\mathbf{DT}}

\nc{\BA}{{\mathbb A}} \nc{\CC}{{\mathbb C}} \nc{\DD}{{\mathbb D}}
\nc{\EE}{{\mathbb E}} \nc{\FF}{{\mathbb F}} \nc{\GG}{{\mathbb G}}
\nc{\HH}{{\mathbb H}} \nc{\LL}{{\mathbb L}} \nc{\NN}{{\mathbb N}}
\nc{\QQ}{{\mathbb Q}} \nc{\RR}{{\mathbb R}} \nc{\BS}{{\mathbb{S}}} \nc{\TT}{{\mathbb T}}
\nc{\VV}{{\mathbb V}} \nc{\ZZ}{{\mathbb Z}}

\nc{\calao}{{\mathcal A}} \nc{\cala}{{\mathcal A}}
\nc{\calc}{{\mathcal C}} \nc{\cald}{{\mathcal D}}
\nc{\cale}{{\mathcal E}} \nc{\calf}{{\mathcal F}}
\nc{\calfr}{{{\mathcal F}^{\,r}}} \nc{\calfo}{{\mathcal F}^0}
\nc{\calfro}{{\mathcal F}^{\,r,0}} \nc{\oF}{\overline{F}}
\nc{\calg}{{\mathcal G}} \nc{\calh}{{\mathcal H}}
\nc{\cali}{{\mathcal I}} \nc{\calj}{{\mathcal J}}
\nc{\call}{{\mathcal L}} \nc{\calm}{{\mathcal M}}
\nc{\caln}{{\mathcal N}} \nc{\calo}{{\mathcal O}}
\nc{\calp}{{\mathcal P}} \nc{\calq}{{\mathcal Q}} \nc{\calr}{{\mathcal R}}
\nc{\calt}{{\mathcal T}} \nc{\caltr}{{\mathcal T}^{\,r}}
\nc{\calu}{{\mathcal U}} \nc{\calv}{{\mathcal V}}
\nc{\calw}{{\mathcal W}} \nc{\calx}{{\mathcal X}}
\nc{\CA}{\mathcal{A}}

\nc{\fraka}{{\mathfrak a}} \nc{\frakB}{{\mathfrak B}}
\nc{\frakb}{{\mathfrak b}} \nc{\frakd}{{\mathfrak d}}
\nc{\oD}{\overline{D}}
\nc{\frakF}{{\mathfrak F}} \nc{\frakg}{{\mathfrak g}}
\nc{\frakm}{{\mathfrak m}} \nc{\frakM}{{\mathfrak M}}
\nc{\frakMo}{{\mathfrak M}^0} \nc{\frakp}{{\mathfrak p}}
\nc{\frakS}{{\mathfrak S}} \nc{\frakSo}{{\mathfrak S}^0}
\nc{\fraks}{{\mathfrak s}} \nc{\os}{\overline{\fraks}}
\nc{\frakT}{{\mathfrak T}}
\nc{\oT}{\overline{T}}
\nc{\frakX}{{\mathfrak X}} \nc{\frakXo}{{\mathfrak X}^0}
\nc{\frakx}{{\mathbf x}}
\nc{\frakTx}{\frakT}      
\nc{\frakTa}{\frakT^a}        
\nc{\frakTxo}{\frakTx^0}   
\nc{\caltao}{\calt^{a,0}}   
\nc{\ox}{\overline{\frakx}} \nc{\fraky}{{\mathfrak y}}
\nc{\frakz}{{\mathfrak z}} \nc{\oX}{\overline{X}}

\font\cyr=wncyr10

\nc{\redtext}[1]{\textcolor{red}{#1}}

\begin{document}

\title{3-Lie bialgebras and 3-Lie classical Yang-Baxter equations in low dimensions}

\author{Chengyu Du}
\address{Chern Institute of Mathematics\& LPMC, Nankai University, Tianjin 300071, China} \email{dcystory@163.com}

\author{Chengming Bai}
\address{Chern Institute of Mathematics \& LPMC, Nankai University, Tianjin 300071, China}
         \email{baicm@nankai.edu.cn}

\author{Li Guo}
\address{Department of Mathematics and Computer Science,
         Rutgers University,
         Newark, NJ 07102}
\email{liguo@rutgers.edu}

\date{\today}

\begin{abstract}
In this paper, we give some low-dimensional examples of local cocycle 3-Lie bialgebras and double construction 3-Lie bialgebras which were introduced in the study of the classical Yang-Baxter equation and Manin triples for 3-Lie algebras. We give an explicit and practical formula to compute the skew-symmetric solutions of the 3-Lie classical Yang-Baxter equation (CYBE). As an illustration, we obtain all skew-symmetric solutions of the 3-Lie CYBE in complex 3-Lie algebras of dimension 3 and 4  and then the induced
local cocycle 3-Lie bialgebras. On the other hand, we classify the
double construction 3-Lie bialgebras for complex 3-Lie algebras in dimensions 3 and 4 and then give the corresponding 8-dimensional
pseudo-metric 3-Lie algebras.
\end{abstract}

\subjclass[2010]{16T10, 16T25, 15A75, 17A30, 17B62, 81T30}

\keywords{bialgebra, 3-Lie algebra, 3-Lie
bialgebra, classical Yang-Baxter equation,  Manin triple}

\maketitle

\tableofcontents
\numberwithin{equation}{section}
\allowdisplaybreaks

\section{Introduction}

Lie algebras have been generalized to higher arities as $n$-Lie
algebras (\cite{Filippov,Kas,Ling}), which have connections with
several fields of mathematics and physics. For example, the
algebraic structure of $n$-Lie algebras correspond to Nambu
mechanics (\cite{Awata,Gautheron,Nambu,Tak}). As a special case of
$n$-Lie algebras, 3-Lie algebras play an important role in string
theory (\cite{Bagger,Figu1,Gus,Ho1,Papa}). As an instance, the
structure of 3-Lie algebras applied in the study of the
supersymmetry and gauge symmetry transformations of the
world-volume theory of multiple coincident M2-branes. In
particular, the metric 3-Lie algebras, or more generally, the
3-Lie algebras with invariant symmetric bilinear forms even
attract more attentions in physics. In fact, the invariant inner
product of a 3-Lie algebra is very useful in order to obtain the
correct equations of motion for the Begger-Lambert theory from a
Lagrangian that is invariant under certain symmetries. In order to
find some new Bagger-Lambert Lagrangians, it is an approach by
concerning 3-Lie algebras with metrics having signature $(p,q)$,
or with a degenerate invariant symmetric bilinear form. Therefore,
it is worthwhile to find new 3-Lie algebras with invariant
symmetric bilinear forms.

On the other hand, it is interesting to consider the bialgebra
structures of 3-Lie algebras. Giving a 3-Lie algebra
$(A,[\cdot,\cdot,\cdot])$, a coalgebra $(A,\Delta)$ such that
$(A^{\ast},\Delta^{\ast})$ is also a 3-Lie algebra, the most
important part for a bialgebra theory is the compatibility
conditions. As pointed out in \cite{C.Bai}, it is quite common for
an algebraic system to have multiple bialgebra structures that
differ only by their compatibility conditions. A good
compatibility condition is prescribed on one hand by a strong
motivation and potential applications, and on the other hand by a
rich structure theory and effective constructions. Motivated by
the well-known Lie bialgebra theory, the following compatibility
conditions are applied in the construction of the bialgebra theory
for 3-Lie algebras:
\begin{enumerate}
\item the comultiplication $\Delta$ satisfies certain
``derivation" condition;
\item the comultiplication $\Delta$ is a
1-cocycle on $A$;
\item there is a Manin triple $(A\oplus
A^{\ast},~A,~A^{\ast})$.
\end{enumerate}
The above three conditions are equivalent in Lie algebras, but the equivalences  are lost when they are extended in the context of 3-Lie algebras. Hence these conditions lead to the following three approaches respectively.
\begin{enumerate}
\item Based on Condition (a), there is an approach of bialgebra theory for 3-Lie algebras introduced in \cite{R.Bai}. Unfortunately, it is a formal generalization in certain sense and neither a coboundary theory nor the structure on the double space $A\oplus A^*$ is known.
\item Motivated by Condition (b) with certain adjustments on the so-called ``1-cocycles", there is a bialgebra theory for 3-Lie algebras which is called ``local
cocycle 3-Lie algebra" in \cite{C.Bai}.  There is a coboundary theory which leads to the introduction of an analogue of the classical Yang-Baxter equation (CYBE), namely,
3-Lie CYBE. That is, from a skew-symmetric solution $r\in A\otimes A$ of 3-Lie CYBE in a 3-Lie algebra $A$, a local cocycle 3-Lie bialgebra is obtained.
\item Based on Condition (c) with the introduction of an analogue of Manin triple for 3-Lie algebras, a bialgebra theory for 3-Lie bialgebras which is called
``double construction 3-Lie algebra" was given in \cite{C.Bai}. Such a construction naturally provides a pseudo-metric 3-Lie algebra structure over the double space $A\oplus A^{\ast}$ with signature $(n,n)$, where $n=\dim A$, for the aforementioned study of Bagger-Lambert Lagrangians.
\end{enumerate}

In this paper, we continue the study of local cocycle and the double construction 3-Lie bialgebras. The main purpose is to provide examples of such 3-Lie bialgebras systematically. For local cocycle 3-Lie bialgebras, we determine the examples from all skew-symmetric solutions of the 3-Lie CYBE. For double construction 3-Lie bialgebras, we obtain a complete classification. We give an explicit and practical formula to compute the skew-symmetric solutions of 3-Lie CYBE and then as an illustration, we give all skew-symmetric solutions of 3-Lie CYBE in the complex 3-Lie algebras in dimension 3 and 4 whose classification is already known (cf. \cite{BaiR}). Hence the induced local cocycle 3-Lie bialgebras are obtained. Besides, we classify the
double construction 3-Lie bialgebras for the complex 3-Lie algebras in dimensions 3 and 4. As a byproduct, for the non-trivial cases, certain 8-dimensional
pseudo-metric 3-Lie algebras are obtained explicitly. These examples can be regarded as a guide for a further development.

The paper is organized as follows. In Section~\ref{sec:bas}, we
give some elementary facts on 3-Lie algebras, the local cocycle 3-Lie bialgebras, the 3-Lie CYBE and the double
construction 3-Lie bialgebras. In Section~\ref{sec:cybe}, we find
all skew-symmetric solutions of 3-Lie CYBE in the complex 3-Lie
algebras in dimension 3 and 4 and the induced local cocycle 3-Lie
bialgebras are given. In Section 4, we classify the double
construction 3-Lie bialgebras for the complex 3-Lie algebras in
dimensions 3 and 4 and hence give certain corresponding
8-dimensional pseudo-metric 3-Lie algebras.

\section{3-Lie algebras and 3-Lie bialgebras}
\label{sec:bas}

In this section we recall notions and results on 3-Lie algebras and 3-Lie bialgebras which will be needed later in the paper. We follow~\cite{C.Bai} to which we refer the reader for further details.

\subsection{3-Lie algebras}

\begin{defi}\label{th:2.1}                                                                             \emph{(\cite{Filippov})} {\rm A {\bf 3-Lie algebra} is a vector space $A$ with a skew-symmetric linear map (3-Lie bracket) $[\cdot ,\cdot ,\cdot ]:\otimes^3 A\rightarrow A$ such that the following \textbf{Fundamental Identity} holds:
\begin{equation}\label{eq:2.1}
[x_1,x_2,[x_3,x_4,x_5]]=[[x_1,x_2,x_3],x_4,x_5]+[x_3,[x_1,x_2,x_4],x_5]+[x_3,x_4,[x_1,x_2,x_5]],
\end{equation}
for any $x_1,\cdots,x_5\in A$.}
\end{defi}
The fundamental identity could be rewritten with the operator
\begin{equation}\label{eq:2.2}
\textrm{ad}_{x_1,x_2}:A\rightarrow A,~~~~~~\textrm{ad}_{x_1,x_2}x=[x_1,x_2,x]
\end{equation}
in the form as
\begin{equation}\label{eq:2.3}
\textrm{ad}_{x_1,x_2}[x_3,x_4,x_5]=[\textrm{ad}_{x_1,x_2}x_3,x_4,x_5]+[x_3,\textrm{ad}_{x_1,x_2}x_4,x_5]+[x_3,x_4,\textrm{ad}_{x_1,x_2}x_5].
\end{equation}
\begin{defi}\label{th:2.2}                                                                        \emph{(\cite{Dzhu,Kas})} {\rm Let $V$ be a vector space. A {\bf representation of a 3-Lie algebra} $A$ on $V$ is a skew-symmetric linear map $\rho:\otimes^2 A\rightarrow gl(V)$ such that for any $x_1,x_2,x_3,x_4\in A$,}
\begin{eqnarray*}
&(i)\; \rho(x_1,x_2)\rho(x_3,x_4)-\rho(x_3,x_4)\rho(x_1,x_2)=\rho([x_1,x_2,x_3],x_4)-\rho([x_1,x_2,x_4],x_3);\\
&(ii)\; \rho([x_1,x_2,x_3],x_4)=\rho(x_1,x_2)\rho(x_3,x_4)+\rho(x_2,x_3)\rho(x_1,x_4)+\rho(x_3,x_1)\rho(x_2,x_4).
\end{eqnarray*}
\end{defi}

Let $(V,\rho)$ be a representation of a $3$-Lie algebra $A$. Define $\rho^{\ast}:\otimes^2 A\longrightarrow\mathfrak{gl}(V^{\ast})$
by
\begin{equation}\label{eq:dual}
\langle \rho^{\ast}(x_1,x_2) \alpha, v\rangle =-\langle \alpha,\rho(x_1,x_2) v\rangle,\quad \forall \alpha\in V^{\ast},~x_1,x_2\in A,~ v\in V.
\end{equation}
\begin{pro}
With the above notations, $(V^{\ast},\rho^{\ast})$ is a representation of $A$, called the dual representation.
\end{pro}

\begin{ex}{\rm Let $A$ be a 3-Lie algebra. The linear map ${\rm ad}:\otimes^2A\rightarrow \frak g\frak l(A)$ with $x_1,x_2\rightarrow {\rm ad}_{x_1,x_2}$ for any $x_1,x_2\in A$
defines a representation $(A, {\rm ad})$ which is called the {\bf adjoint representation} of $A$, where ${\rm ad}_{x_1,x_2}$ is given by Eq.~\eqref{eq:2.2}.
The dual representation $(A^*,{\rm ad}^*)$ of the adjoint representation $(A,{\rm ad})$ of a 3-Lie algebra $A$ is called
the {\bf coadjoint representation}.
}\end{ex}

The classification of complex 3-Lie algebras in dimension 3 and 4 has been known (cf. \cite{BaiR}).

\begin{pro}\label{th:2.3}
There is a unique non-trivial 3-dimensional complex 3-Lie algebra. It has a basis $\{e_1,e_2,e_3\}$ with respect to which the non-zero product is given by
\begin{equation*}
[e_1,e_2,e_3]=e_1.
\end{equation*}
\end{pro}

\begin{pro}\label{th:2.4}
Let $A$ be a non-trivial 4-dimensional complex 3-Lie algebra. Then $A$ has a basis $\{e_1,e_2,e_3,e_4\}$ with respect to which the non-zero product of the 3-Lie algebra is given by one of the following.
\begin{eqnarray*}
&~&(1)~[e_1,e_2,e_3]=e_4,[e_1,e_2,e_4]=e_3,[e_1,e_3,e_4]=e_2,[e_2,e_3,e_4]=e_1;\\
&~&(2)~[e_1,e_2,e_3]=e_1;\\
&~&(3)~[e_2,e_3,e_4]=e_1;\\
&~&(4)~[e_2,e_3,e_4]=e_1,[e_1,e_3,e_4]=e_2;\\
&~&(5)~[e_2,e_3,e_4]=e_2,[e_1,e_3,e_4]=e_1;\\
&~&(6)~[e_2,e_3,e_4]=\alpha e_1+e_2,~\alpha\neq 0,[e_1,e_3,e_4]=e_2.\\
&~&(7)~[e_1,e_2,e_4]=e_3,[e_1,e_3,e_4]=e_2,[e_2,e_3,e_4]=e_1.
\end{eqnarray*}
\end{pro}

\subsection{Local cocycle 3-Lie bialgebras and the 3-Lie classical Yang-Baxter equation}
Most of the facts in this subsection and next subsection can be found in \cite{C.Bai}.

\begin{defi}\label{th:2.5}                                                                    {\rm Let $A$ be a 3-Lie algebra and $(V,\rho)$ be a representation of $A$. A linear map $f:A\rightarrow V$ is called a {\bf 1-cocycle} of $A$ associated to $(V,\rho)$ if it satisfies}
\begin{equation*}
f([x_1,x_2,x_3])=\rho(x_1,x_2)f(x_3)+\rho(x_2,x_3)f(x_1)+\rho(x_3,x_1)f(x_2),\;\;\forall x_1,x_2,x_3\in A.
\end{equation*}
\end{defi}
\begin{defi}   {\rm                                                                    A {\bf local cocycle 3-Lie bialgebra} is a pair $(A,\Delta)$, where $A$ is a 3-Lie algebra, and $\Delta=\Delta_1+\Delta_2+\Delta_3:A\rightarrow A\otimes A\otimes A$ is a linear map, such that $\Delta^{\ast}:A^{\ast}\otimes A^{\ast}\otimes A^{\ast}\rightarrow A^{\ast}$ defines a 3-Lie algebra structure on $A^{\ast}$, and the following conditions are satisfied:
\begin{eqnarray*}
&(1)~\Delta_1~\text{is a 1-cocycle associated to the representation }(A\otimes A\otimes A, \emph{\rm ad}\otimes \id \otimes \id );\\
&(2)~\Delta_2~\text{is a 1-cocycle associated to the representation }(A\otimes A\otimes A, \id\otimes \emph{\rm ad}\otimes \id );\\
&(3)~\Delta_3~\text{is a 1-cocycle associated to the representation }(A\otimes A\otimes A, \id\otimes \id \otimes \emph{\rm ad}).
\end{eqnarray*}}
\end{defi}

In order to define the 3-Lie classical Yang-Baxter equation, we first give some necessary notations.
Let $A$ be a 3-Lie algebra and $r=\sum_i x_i\otimes y_i\in A\otimes A$. For any $1\leq p\neq q\leq 4$, define an inclusion $\cdot_{pq}:\otimes^2A\longrightarrow \otimes^4 A$ by
$$
r_{pq}:=\sum_i z_{i1}\otimes\cdots\otimes z_{in},\quad \text{ where } z_{ij}=\left\{\begin{array}{ll} x_i,& j=p,\\ y_i, & j=q, \\ 1, & i\neq p, q,\end{array} \right.
$$
where 1 is a symbol playing a similar role of the unit. Hence define $[[r,r,r]]\in \otimes^4 A$ by
\begin{eqnarray}
\label{eq:rrr}[[r,r,r]]&:=&[r_{12},r_{13},r_{14}]+[r_{12},r_{23},r_{24}]+[r_{13},r_{23},r_{34}]+[r_{14},r_{24},r_{34}]\\
\nonumber&=&\sum_{i,j,k}\big([x_i,x_j,x_k]\otimes y_i\otimes y_j\otimes y_k+x_i\otimes [y_i,x_j,x_k]\otimes y_j\otimes y_k\\
\nonumber&&+ x_i\otimes x_j\otimes [y_i, y_j,x_k]\otimes y_k+ x_i\otimes x_j\otimes x_k\otimes [y_i,y_j,y_k]\big).
\end{eqnarray}

\begin{defi}{\rm
Let $A$ be a $3$-Lie algebra and $r\in A\otimes A$. The equation
$$[[r,r,r]]=0$$
is called the {\bf $3$-Lie classical Yang-Baxter equation (3-Lie CYBE)}.
\mlabel{defi:3cybe}}
\end{defi}

\begin{lem}
Let $A$ be a 3-Lie algebra and $r=\sum_i x_i\otimes y_i\in A\otimes A$. Set
\begin{equation}\label{eq:delta123}\left\{\begin{array}{ccc}
\Delta_1(x)&:=&\sum_{i,j} [x,x_i,x_j]\otimes y_j\otimes y_i;\\
\Delta_2(x)&:=&\sum_{i,j} y_i\otimes [x,x_i,x_j]\otimes y_j;\\
\Delta_3(x)&:=&\sum_{i,j} y_j\otimes y_i\otimes [x,x_i,x_j],
\end{array}\right.
\end{equation}
where $x\in A$. Then
\begin{enumerate}
\item[\rm (1)] $\Delta_1$ is a $1$-cocycle associated to the representation $(A\otimes A\otimes A, {\rm ad}\otimes \id \otimes \id )$;
\item[\rm (2)] $\Delta_2$ is a $1$-cocycle associated to the representation $(A\otimes A\otimes A,\id\otimes {\rm ad}\otimes \id )$;
\item[\rm (3)] $\Delta_3$ is a $1$-cocycle associated to the representation $(A\otimes A\otimes A,\id\otimes \id \otimes {\rm ad})$.
\end{enumerate}
Moreover,
$\Delta^*:A^*\otimes A^*\otimes A^*\rightarrow A^*$ defines a
skew-symmetric operation, where $\Delta=\Delta_1+\Delta_2+\Delta_3$.
\end{lem}

As is well-known, a skew-symmetric solution of the CYBE in a Lie algebra gives a Lie bialgebra. As its 3-Lie algebra
analogue, we have

\begin{thm}
Let $A$ be a $3$-Lie algebra and let $r\in A\otimes A$ be a skew-symmetric solution of the 3-Lie CYBE:
$$[[r,r,r]]=0.$$
Define  $\Delta:=\Delta_1+\Delta_2+\Delta_3:
A\rightarrow A\otimes A\otimes A$, where
$\Delta_1, \Delta_2, \Delta_3$ are induced by $r$ as in Eq.~\eqref{eq:delta123}.
Then $\Delta^*$ defines a $3$-Lie algebra structure on $A^*$.   Furthermore, $(A,\Delta)$ is a local cocycle 3-Lie bialgebra.
\label{thm:ybe}
\end{thm}

\subsection{Double construction 3-Lie bialgebras}
We end this preparational section with recalling the notion of a double construction 3-Lie bialgebra and its related Minin triple.

\begin{defi}\label{doulbe}
{\rm Let $A$ be a 3-Lie algebra and $\Delta:A\rightarrow A\otimes A\otimes A$ a linear map. Suppose that $\Delta^{\ast}:A^{\ast}\otimes A^{\ast}\otimes A^{\ast}\to A^{\ast}$ defines a 3-Lie algebra structure on $A^{\ast}$. If for all $x,y,z\in A$, $\Delta$ satisfies the following conditions,
\begin{equation}\label{eq:2.6}
\Delta([x,y,z])=(\id\otimes \id \otimes \emph{\rm ad}_{y,z})\Delta(x)+(\id\otimes \id \otimes \emph{\rm ad}_{z,x})\Delta(y)+(\id\otimes \id \otimes \emph{\rm ad}_{x,y})\Delta(z),
\end{equation}
\begin{equation}\label{eq:2.7}
\Delta([x,y,z])=(\id\otimes \id \otimes \emph{\rm ad}_{y,z})\Delta(x)+(\id\otimes \emph{\rm ad}_{y,z}\otimes \id )\Delta(x)+(\emph{\rm ad}_{y,z}\otimes \id \otimes \id )\Delta(x),
\end{equation}
then we call $(A,\Delta)$ a {\bf  double construction 3-Lie bialgebra.}}
\end{defi}

\begin{defi}
{\rm
Let $A$ be a $3$-Lie algebra. A bilinear form $(\cdot,\cdot)_A$  on $A$ is called {\bf invariant} if it satisfies
\begin{equation}
 ([x_1,x_2,x_3],x_4)_A+([x_1,x_2,x_4],x_3)_A=0,\quad\forall  x_1,x_2,x_3,x_4\in A.
\end{equation}
A $3$-Lie algebra $A$ is called {\bf pseudo-metric} if there is a nondegenerate symmetric invariant bilinear form on $A$.
}
\end{defi}

\begin{defi}\label{defi:Manin}
{\rm
A {\bf Manin triple of $3$-Lie algebras} consists of a pseudo-metric $3$-Lie algebra $(\mathcal{A},(\cdot,\cdot)_\mathcal{A})$ and $3$-Lie algebras $A_1, A_2$   such that
\begin{enumerate}
\item[\rm (1)] $A_1,A_2$ are isotropic $3$-Lie subalgebras of $\mathcal{A}$;
\item[\rm (2)] $\mathcal{A}=A_1\oplus A_2$ as the direct sum of vector spaces;
\item[\rm (3)] For all $x_1,y_1\in A_1 $ and $x_2,y_2\in A_2$, we have $\mathrm{pr}_1[x_1,y_1,x_2]=0$ and $\mathrm{pr}_2[x_2,y_2,x_1]=0$, where $\mathrm{pr}_1$ and $\mathrm{pr}_2$ are the projections from $A_1\oplus A_2$ to $A_1$ and $A_2$ respectively.
\end{enumerate}
}
\end{defi}

Let $(A,[\cdot,\cdot,\cdot])$ and $(A^{\ast},[\cdot,\cdot,\cdot]^{\ast})$ be  3-Lie algebras. On $A\oplus A^{\ast}$, there is a natural nondegenerate symmetric bilinear form $(\cdot,\cdot)_+$  given by
\begin{equation} \label{eq:bf}
( x+\xi, y+\eta)_+=\langle x, \eta\rangle+\langle \xi,y\rangle,\;\;\forall  x,y\in A, \xi,\eta\in A^{\ast}.
\end{equation}
There is also a bracket operation
$[\cdot,\cdot,\cdot]_{A\oplus A^{\ast}}$ on $A\oplus A^{\ast}$ given by
\begin{eqnarray}
 \nonumber [x+\xi,y+\eta,z+\gamma]_{A\oplus A^{\ast}}&=&[x,y,z]+\mathrm{ad}_{x,y}^{\ast}\gamma+\mathrm{ad}_{y,z}^{\ast}\xi+\mathrm{ad}_{z,x}^{\ast}\eta\\
\label{eq:formularAA*}  &&+\mathfrak{ad}_{\xi,\eta}^{\ast}z+\mathfrak{ad}_{\eta,\gamma}^{\ast}x+\mathfrak{ad}_{\gamma,\xi}^{\ast}y+[\xi,\eta,\gamma]^{\ast},
\end{eqnarray}
where $\mathrm{ad}^{\ast}$ and $\mathfrak{ad}^{\ast}$ are the coadjoint representations of $A$ and $A^{\ast}$ on $A^{\ast}$ and $A$ respectively. Note that the bracket operation $[\cdot,\cdot,\cdot]_{A\oplus A^{\ast}}$ is naturally invariant with respect to the symmetric bilinear form $(\cdot,\cdot)_+$, and satisfies Condition~(3) in Definition~\ref{defi:Manin}. If $(A\oplus A^{\ast},[\cdot,\cdot,\cdot]_{A\oplus A^{\ast}})$ is a 3-Lie algebra, then obviously $A$ and $A^{\ast}$ are isotropic subalgebras. Consequently, $((A\oplus A^{\ast},(\cdot,\cdot)_+),A, A^{\ast})$ is a Manin triple, which is called {\bf the standard Manin triple of 3-Lie algebras.}

\begin{thm}\label{thm:relations}
Let $A$ be a $3$-Lie algebra and $\Delta:A\rightarrow A\otimes A\otimes A$  a linear map. Suppose that $\Delta^*:A^*\otimes A^*\otimes A^*\rightarrow A^*$
defines a $3$-Lie algebra structure on $A^*$. Then  $(A,\Delta)$ is a double construction 3-Lie bialgebra if and only if
$((A\oplus A^*,(\cdot,\cdot)_+),A,A^*)$ is a standard Manin triple, where the bilinear form $(\cdot,\cdot)_+$ and the $3$-Lie bracket $[\cdot,\cdot,\cdot]_{A\oplus A^*}$ are given by Eqs.~\eqref{eq:bf} and \eqref{eq:formularAA*} respectively.
\end{thm}

\section{Skew-symmetric solutions of the 3-Lie CYBE and local cocycle 3-Lie bialgebras}
\label{sec:cybe}

In this section, we obtain a computable formula of the 3-Lie CYBE
and apply it to obtain all skew-symmetric solutions of the 3-Lie
CYBE in the complex 3-Lie algebras in dimensions 3 and 4. We then
obtain the local cocycle 3-Lie bialgebras induced from these
solutions.

\subsection{Notational simplification of the 3-Lie CYBE}

Let $\{e_1,\cdots,e_n\}$ be a basis of $A$. Set
\begin{equation}\label{eq:3.2}
    r=\sum_{i,j}a^{ij}e_i\otimes e_j=\sum_ie_i \otimes  (\sum_ja^{ij} e_j) \in A\otimes A. ~~
\end{equation}
Then
\begin{eqnarray*}
        [r_{12},r_{13},r_{14}]&=&\sum_{i,j,k,p,q,r}a^{ip}a^{jq}a^{kr}[e_i,e_j,e_k]\otimes e_p\otimes e_q\otimes e_r\\
    &=&\sum_{p,q,r}\big(\sum_{i,j,k}a^{ip}a^{jq}a^{kr}[e_i,e_j,e_k]\otimes e_p\otimes e_q\otimes e_r\big).
\end{eqnarray*}
If any two of $\{i,j,k\}$ are equal, then $[e_i,e_j,e_k]=0$. So we
can assume that $\{i,j,k\}$ are distinct in the sum.

Let $S_3$ denote the symmetric group
of order 3. In the following, we let $\sigma\in S_3$ act on $\{i,j,k\}$ by permuting the three locations. So denoting $(i_1,i_2,i_3):=(i,j,k)$, we define $\{\sigma(i),\sigma(j),\sigma(k)\} =\{i_{\sigma(1)},i_{\sigma(2)},i_{\sigma(3)}\}$. This applies even if $i,j,k$ are not distinct. Then for fixed $\{p,q,r\}$,  we have
    \begin{eqnarray*}
&&\sum_{i,j,k}a^{ip}a^{jq}a^{kr}[e_i,e_j,e_k]\otimes e_p\otimes
e_q\otimes e_r\\
&&=\sum_{i<j<k}\sum_{\sigma\in S_3}a^{\sigma(i)p}a^{\sigma(j)q}a^{\sigma(k)r}[e_{\sigma(i)},e_{\sigma(j)},e_{\sigma(k)}]\otimes e_p\otimes e_q\otimes e_r\\
        &&=\sum_{i<j<k}\big(\sum_{\sigma\in S_3}\mathrm{sgn}(\sigma)a^{\sigma(i)p}a^{\sigma(j)q}a^{\sigma(k)r}\big)[e_i,e_j,e_k]\otimes e_p\otimes e_q\otimes
        e_r.
    \end{eqnarray*}
Therefore $[r_{12},r_{13},r_{14}]$ is rewritten as
    \begin{equation*}
        [r_{12},r_{13},r_{14}]~=~\sum_{p,q,r}\sum_{i<j<k}(\sum_{\sigma\in S_3}\mathrm{sgn}(\sigma)a^{\sigma(i)p}a^{\sigma(j)q}a^{\sigma(k)r})[e_i,e_j,e_k]\otimes e_p\otimes e_q\otimes e_r.
    \end{equation*}
Similarly, we have
\begin{eqnarray*}
        ~[r_{12},r_{23},r_{24}]
    &=&\sum_{p,q,r}\sum_{i<j<k}(\sum_{\sigma\in S_3}\mathrm{sgn}(\sigma)a^{p\sigma(i)}a^{\sigma(j)q}a^{\sigma(k)r})e_p\otimes[e_i,e_j,e_k]\otimes e_q\otimes e_r,\\
        ~[r_{13},r_{23},r_{34}]
    &=&\sum_{p,q,r}\sum_{i<j<k}(\sum_{\sigma\in S_3}\mathrm{sgn}(\sigma)a^{p\sigma(i)}a^{q\sigma(j)}a^{\sigma(k)r})e_p\otimes e_q\otimes[e_i,e_j,e_k]\otimes e_r,\\
        ~[r_{14},r_{24},r_{34}]
    &=&\sum_{p,q,r}\sum_{i<j<k}(\sum_{\sigma\in S_3}\mathrm{sgn}(\sigma)a^{p\sigma(i)}a^{q\sigma(j)}a^{r\sigma(k)})e_p\otimes e_q\otimes e_r\otimes[e_i,e_j,e_k].\\
\end{eqnarray*}
For all $i< j< k, p,q,r$, set
\begin{align}
    M_{pqr}^{ijk}(1)&=(\sum_{\sigma\in S_3}\mathrm{sgn}(\sigma)a^{\sigma(i)p}a^{\sigma(j)q}a^{\sigma(k)r})[e_i,e_j,e_k]\otimes e_p\otimes e_q\otimes e_r, \label{eq:m1}\\
    M_{pqr}^{ijk}(2)&=(\sum_{\sigma\in S_3}\mathrm{sgn}(\sigma)a^{p\sigma(i)}a^{\sigma(j)q}a^{\sigma(k)r})e_p\otimes[e_i,e_j,e_k]\otimes e_q\otimes e_r,\label{eq:m2}\\
    M_{pqr}^{ijk}(3)&=(\sum_{\sigma\in S_3}\mathrm{sgn}(\sigma)a^{p\sigma(i)}a^{q\sigma(j)}a^{\sigma(k)r})e_p\otimes e_q\otimes[e_i,e_j,e_k]\otimes e_r,\label{eq:m3}\\
    M_{pqr}^{ijk}(4)&=(\sum_{\sigma\in S_3}\mathrm{sgn}(\sigma)a^{p\sigma(i)}a^{q\sigma(j)}a^{r\sigma(k)})e_p\otimes e_q\otimes e_r\otimes[e_i,e_j,e_k].\label{eq:m4}
\end{align}
Thus, we have
\begin{align}
    ~[r_{12},r_{13},r_{14}]&=\sum_{p,q,r}\sum_{i<j<k}M_{pqr}^{ijk}(1),\label{r1.1}\\
    ~[r_{12},r_{23},r_{24}]&=\sum_{p,q,r}\sum_{i<j<k}M_{pqr}^{ijk}(2),\label{r1.2}\\
    ~[r_{13},r_{23},r_{34}]&=\sum_{p,q,r}\sum_{i<j<k}M_{pqr}^{ijk}(3),\label{r1.3}\\
    ~[r_{14},r_{24},r_{34}]&=\sum_{p,q,r}\sum_{i<j<k}M_{pqr}^{ijk}(4).\label{r1.4}
\end{align}
Moreover, it is obvious that $M_{pqr}^{ijk}(m)$ is
invariant under the permutations on $\{i,j,k\}$, i.e.,
\begin{equation}
    M^{ijk}_{pqr}(m)=M^{\sigma(i)\sigma(j)\sigma(k)}_{pqr}(m),
    \;\; \forall p,q,r, i<j<k, m=1,2,3,4.
\end{equation}

Let $V$ be a vector space and let $\wedge$ denote the exterior product.  For example,
$$x\wedge y=x\otimes y-y\otimes x,\;\; x_i\wedge x_j\wedge x_k=\sum_{\sigma\in
S_3}\textrm{sgn}(\sigma)x_{\sigma(i)}\otimes
x_{\sigma(j)}\otimes x_{\sigma(k)},\;\;\forall x,y, x_i, x_j, x_k\in V.$$
Let $\wedge^k(V)$ denote the $k$-th exterior power of $V$.

\subsection{Skew-symmetric solutions of the 3-Lie CYBE}

We now give a simplified formula for the 3-Lie CYBE $[[r,r,r]]=0$ when $r=\sum\limits_{i,j}a^{ij}e_i\otimes e_j$ is
skew-symmetric, i.e., $a^{ij}=-a^{ji}$.

\begin{thm}\label{th:3.16}  Let $A$ be a 3-Lie algebra with a
basis $\{ e_1,\cdots,e_n\}$. Let the ternary operation be given by
\begin{equation}\label{eq:3.15}
[e_i,e_j,e_k]=\sum_{m}T_{ijk}^{m} e_m.
\end{equation}
Suppose that $r=\sum\limits_{i,j}a^{ij}e_i\otimes e_j\in A\otimes
A$ is skew-symmetric. Then
\begin{equation}\label{eq:3.22}
[[r,r,r]]=\sum_{p<q<r}\sum_{i<j<k}\sum_l D^{ijk}_{pqr}T_{ijk}^{l}
e_l\wedge e_p\wedge e_q\wedge e_r,
\end{equation}
where
\begin{equation*}
    D^{ijk}_{pqr}:=\sum_{\sigma\in S_3}\mathrm{sgn}(\sigma)a^{\sigma(i)p}a^{\sigma(j)q}a^{\sigma(k)r},\;\;\forall i,j,k,p,q,r=1,\cdots,n.
\end{equation*}
\end{thm}

The theorem has a direct consequence.

\begin{cor}\label{co:3.2}
Let $A$ be a 3-Lie algebra. If $r\in \wedge^2 (A)$, then $[[r,r,r]]\in \wedge^4(A)$.
\end{cor}

\begin{rmk}
{\rm
This corollary can be regarded as a generalization of the following result on Lie algebras in the context of 3-Lie algebras:
for a Lie algebra $\frak g$, if $r\in \wedge^2 (\frak g)$, then
$$
[[r,r]] =
[r_{12},r_{13}]+[r_{12},r_{23}]+[r_{13},r_{23}]\in \wedge^3 (\frak
g).$$
}
\end{rmk}

We will prove \text{Theorem~\ref{th:3.16}} in several steps.

First by the skew-symmetry of $r$, Eqs.~(\ref{eq:m1}) -- (\ref{eq:m2}) become
    \begin{align}
        M_{pqr}^{ijk}(1)&=D^{ijk}_{pqr}[e_i,e_j,e_k]\otimes e_p\otimes e_q\otimes e_r,\label{m1.1}\\
        M_{pqr}^{ijk}(2)&=-D^{ijk}_{pqr}e_p\otimes[e_i,e_j,e_k]\otimes e_q\otimes e_r,\label{m1.2}\\
        M_{pqr}^{ijk}(3)&=D^{ijk}_{pqr}e_p\otimes e_q\otimes[e_i,e_j,e_k]\otimes e_r,\label{m1.3}\\
        M_{pqr}^{ijk}(4)&=-D^{ijk}_{pqr}e_p\otimes e_q\otimes e_r\otimes[e_i,e_j,e_k].\label{m1.4}
    \end{align}

\begin{lem}\label{le:3.2}  With the notations and conditions as
above. Then
\begin{enumerate}
\item $D^{ijk}_{pqr}=\mathrm{sgn}(\tau )D^{\tau (i)\tau (j)\tau (k)}_{pqr}$ for any
         $\tau \in S_3$.
\item $D^{ijk}_{pqr}=-D^{pqr}_{ijk}$.
\item  $
D^{ijk}_{pqr}=\mathrm{sgn}(\ttau)D^{ijk}_{\ttau(p)\ttau(q)\ttau(r)}$
for any $\ttau\in S_3$.
\item If $p, q, r$ are not distinct, then
$M^{ijk}_{pqr}(m)=0$ for any $m=1,2,3,4$.
\item $D_{ijk}^{ijk}=0$.
\end{enumerate}
\end{lem}
    \begin{proof}
      (a)   First for any $\tau \in S_3$, since $S_3\tau=S_3$, we have
        \begin{align*}
            D^{ijk}_{pqr}&=\sum_{\sigma\in S_3}\mathrm{sgn}(\sigma\tau )a^{\sigma\tau (i)p}a^{\sigma\tau (j)q}a^{\sigma\tau (k)r}\\
            &=\mathrm{sgn}(\tau )\sum_{\sigma\in S_3}\mathrm{sgn}(\sigma)a^{\sigma\tau (i)p}a^{\sigma\tau (j)q}a^{\sigma\tau (k)r}\\
            &=\mathrm{sgn}(\tau )D^{\tau (i)\tau (j)\tau (k)}_{pqr}.
        \end{align*}

    (b) In fact, we have
        \begin{align*}
            D^{ijk}_{pqr}&=a^{ip}a^{jq}a^{kr}-a^{ip}a^{kq}a^{jr}+a^{jp}a^{kq}a^{ir}-a^{jp}a^{iq}a^{kr}+a^{kp}a^{iq}a^{jr}-a^{kp}a^{jq}a^{ir},\\
            D^{pqr}_{ijk}&=a^{pi}a^{qj}a^{rk}-a^{pi}a^{rj}a^{qk}+a^{qi}a^{rj}a^{pk}-a^{qi}a^{pj}a^{rk}+a^{ri}a^{pj}a^{qk}-a^{ri}a^{qj}a^{pk}\\
            &=(-1)^3(a^{ip}a^{jq}a^{kr}-a^{ip}a^{jr}a^{kq}+a^{iq}a^{jr}a^{kp}-a^{iq}a^{jp}a^{kr}+a^{ir}a^{jp}a^{kq}-a^{ir}a^{jq}a^{kp})\\
            &=-D^{ijk}_{pqr}.
        \end{align*}
   Hence $D^{ijk}_{pqr}=-D^{pqr}_{ijk}$.

(c) By (a) and (b), we have
        \begin{equation*}
            D^{ijk}_{pqr}=-D^{pqr}_{ijk}=-\mathrm{sgn}(\ttau)D_{ijk}^{\ttau(p)\ttau(q)\ttau(r)}=\mathrm{sgn}(\ttau)D^{ijk}_{\ttau(p)\ttau(q)\ttau(r)}.
        \end{equation*}

 (d) It is a direct consequence due to (c) by taking $\sigma$ to be the transposition exchanging two of $p,q,r$ which are not distinct.

(e) This follows since by (b), $D^{ijk}_{ijk}=-D^{ijk}_{ijk}$ and hence must be zero.

Now the proof of the lemma is completed.
    \end{proof}

By Lemma~\ref{le:3.2}, we can assume that in Eqs.~(\ref{r1.1}) -- (\ref{r1.4}), both $\{i,j,k\}$ and $\{p,q,r\}$ consist of distinct elements.
Then these equations can be simplified to
\begin{align}
[r_{12},r_{13},r_{14}]=\sum_{p<q<r}\sum_{i<j<k}\sum_{\ttau\in \tS_3}M^{ijk}_{\ttau(p)\ttau(q)\ttau(r)}(1),\label{r2.1}\\
[r_{12},r_{23},r_{24}]=\sum_{p<q<r}\sum_{i<j<k}\sum_{\ttau\in \tS_3}M^{ijk}_{\ttau(p)\ttau(q)\ttau(r)}(2),\label{r2.2}\\
[r_{13},r_{23},r_{34}]=\sum_{p<q<r}\sum_{i<j<k}\sum_{\ttau\in \tS_3}M^{ijk}_{\ttau(p)\ttau(q)\ttau(r)}(3),\label{r2.3}\\
[r_{14},r_{24},r_{34}]=\sum_{p<q<r}\sum_{i<j<k}\sum_{\ttau\in \tS_3}M^{ijk}_{\ttau(p)\ttau(q)\ttau(r)}(4).\label{r2.4}
\end{align}

Now we can give the proof of Theorem~\ref{th:3.16}.
\smallskip

\noindent
{\it Proof of Theorem~\ref{th:3.16}.} By Eqs.~(\ref{r2.1}) -- (\ref{r2.4}),  we have
\begin{equation}\label{eq:rrr}
[[r,r,r]]=\sum_m\sum_{p<q<r}\sum_{i<j<k}\sum_{\ttau\in \tS_3}M^{ijk}_{\ttau(p)\ttau(q)\ttau(r)}(m).
\end{equation}
Next we need to show
\begin{equation}\label{eq:3.21}
\sum_m\sum_{\ttau\in \tS_3}M^{ijk}_{\ttau(p)\ttau(q)\ttau(r)}(m)=\sum_l D^{ijk}_{pqr}T_{ijk}^{l}e_l\wedge e_p\wedge e_q\wedge e_r.
\end{equation}
Define an operator $\phi_{pq}:A^{\otimes m}\rightarrow A^{\otimes m}$ by
$$\phi_{pq}(\sum x_1\otimes\cdots\otimes x_p\otimes\cdots\otimes x_q\otimes\cdots\otimes x_m)=\sum x_1\otimes\cdots\otimes x_q\otimes\cdots\otimes x_p\otimes\cdots\otimes x_m.$$
Obviously, $\phi_{pp}$ is the identity. Then by Eqs.~(\ref{m1.1}) -- (\ref{m1.4}), we have
    \begin{align*}
        M_{pqr}^{ijk}(2)=-\phi_{12}M_{pqr}^{ijk}(1),\;\;
        M_{pqr}^{ijk}(3)=\phi_{23}\phi_{12}M_{pqr}^{ijk}(1),\;\;
        M_{pqr}^{ijk}(4)=-\phi_{34}\phi_{23}\phi_{12}M_{pqr}^{ijk}(1).
    \end{align*}
Moreover, by \text{Lemma \ref{le:3.2}} ~ (c), we have
    \begin{align*}
        M_{\tau(p)\tau(q)\tau(r)}^{ijk}(1)&=\sum_{\ttau\in \tS_3}D^{ijk}_{\ttau(p)\ttau(q)\ttau(r)}[e_i,e_j,e_k]\otimes e_{\ttau(p)}\otimes e_{\ttau(q)}\otimes e_{\ttau(r)}\\
                        &=\sum_l\sum_{\ttau\in \tS_3}D^{ijk}_{\ttau(p)\ttau(q)\ttau(r)}T_{ijk}^l e_l\otimes e_{\ttau(p)}\otimes e_{\ttau(q)}\otimes e_{\ttau(r)}\\
                        &=\sum_l\sum_{\ttau\in \tS_3}\mathrm{sgn}(\ttau)D^{ijk}_{pqr}T_{ijk}^l e_l\otimes e_{\ttau(p)}\otimes e_{\ttau(q)}\otimes e_{\ttau(r)}\\
                        &=\sum_l D^{ijk}_{pqr}T_{ijk}^l e_l\otimes\big(\sum_{\ttau\in \tS_3}\mathrm{sgn}(\ttau)e_{\ttau(p)}\otimes e_{\ttau(q)}\otimes e_{\ttau(r)}\big)\\
                        &=\sum_l D^{ijk}_{pqr}T_{ijk}^l e_l\otimes\big(e_p\wedge e_q\wedge e_r).
    \end{align*}
Therefore
    \begin{equation*}
        \sum_m\sum_{\ttau\in \tS_3}M^{ijk}_{\ttau(p)\ttau(q)\ttau(r)}(m)=\sum_l D^{ijk}_{pqr}T_{ijk}^{l}(\phi_{11}-\phi_{12}+\phi_{23}\phi_{12}-\phi_{34}\phi_{23}\phi_{12})e_l\otimes(e_p\wedge e_q\wedge e_r).
    \end{equation*}
It remains to prove
 $$(\phi_{11}-\phi_{12}+\phi_{23}\phi_{12}-\phi_{34}\phi_{23}\phi_{12})e_l\otimes(e_p\wedge e_q\wedge e_r)=e_l\wedge e_p\wedge e_q\wedge e_r.$$ In fact, the right hand side gives
    \begin{align*}
        e_l\wedge e_p\wedge e_q\wedge e_r=&\sum_{\ttau\in \tS_3}\mathrm{sgn}(\ttau)e_l\otimes e_{\ttau(p)}\otimes e_{\ttau(q)}\otimes e_{\ttau(r)}-\sum_{\ttau\in \tS_3}\mathrm{sgn}(\ttau)e_{\ttau(p)}\otimes e_l\otimes e_{\ttau(q)}\otimes e_{\ttau(r)}\\
            &+\sum_{\ttau\in \tS_3}\mathrm{sgn}(\ttau)e_{\ttau(p)}\otimes e_{\ttau(q)}\otimes e_l\otimes e_{\ttau(r)}-\sum_{\ttau\in \tS_3}\mathrm{sgn}(\ttau)e_{\ttau(p)}\otimes e_{\ttau(q)}\otimes e_{\ttau(r)}\otimes e_l\\
             =&(\phi_{11}-\phi_{12}+\phi_{23}\phi_{12}-\phi_{34}\phi_{23}\phi_{12})e_l\otimes\big(\sum_{\ttau\in \tS_3}\mathrm{sgn}(\ttau)e_{\ttau(p)}\otimes e_{\ttau(q)}\otimes e_{\ttau(r)}\big)\\
             =&(\phi_{11}-\phi_{12}+\phi_{23}\phi_{12}-\phi_{34}\phi_{23}\phi_{12})e_l\otimes(e_p\wedge e_q\wedge e_r)
    \end{align*}
Hence \text{Theorem~\ref{th:3.16}} holds. \hfill $\Box$

\subsection{Skew-symmetric solutions of the 3-Lie CYBE in the complex 3-Lie algebras in dimensions 3 and
4}

We first consider the dimension 3 case.

\begin{thm}\label{th:dim3} Let $A$ be a 3-dimensional 3-Lie algebra. Then for any $r\in\wedge^2(A)$,
$[[r,r,r]]=0$. That is, any $r\in \wedge^2(A)$ is a
solution of 3-Lie CYBE in a 3-dimensional 3-Lie algebra.
\end{thm}
\begin{proof}
By Corollary~\ref{co:3.2}, $[[r,r,r]]\in \wedge^4(A)$. On the other hand,  since $\mathrm{dim}~A=3$,  any element in $\wedge^4(A)$ is zero. Therefore $[[r,r,r]]=0$.
\end{proof}

Next let $A$ be a 4-dimensional 3-Lie algebra with a
basis $\{ e_1,\cdots,e_4\}$. Then we have
\begin{equation}\label{eq:4.1}
[e_i,e_j,e_k]=\sum_{m=1}^4T_{ijk}^{m} e_m, \forall i,j,k=1,\cdots 4,
\end{equation}
for constants $T_{ijk}^m$.

\begin{lem}\label{le:3.8} With the notations and conditions as
above, assume that $T_{ijk}^l\neq0$ only if $i,j,k,l$ are distinct. Then for any skew-symmetric $r\in  \wedge^2(A)$,
$[[r,r,r]]=0$.
\end{lem}
    \begin{proof}
First $e_l\wedge e_p\wedge e_q\wedge e_r\in \wedge^4 (A)$ is not zero precisely when $l,p,q,r$ are distinct. By the assumption, the indices of a nonzero $T_{ijk}^{l}$ in Eq.~ (\ref{eq:3.22}) are also distinct. But $i,j,k,p,q,r,l\in \{1,2,3,4\}$ and $i<j<k,~p<q<r$ in Eq.~(\ref{eq:3.22}). Thus we have  $i=p,~j=q,~k=r$. Then Eq.~(\ref{eq:3.22}) gives
$$[[r,r,r]]=\sum_l\sum_{p<q<r}D^{pqr}_{pqr}e_l\wedge e_p\wedge e_q\wedge e_r.$$
        By \text{Lemma~\ref{le:3.2}} (e), $[[r,r,r]]=0$.
    \end{proof}

\begin{thm}\label{th:dim4}
Let $A$ be a 4-dimensional 3-Lie algebra. If $A$ is one of the complex 3-Lie algebras
of Cases~(1), (3), (4),and (7) given in Proposition~\ref{th:2.4},  then any skew-symmetric $r\in \wedge^2(A)$
satisfies $[[r,r,r]]=0$.
\end{thm}

    \begin{proof}
The proof follows directly from \text{Lemma~\ref{le:3.8}}.
    \end{proof}

\begin{thm}\label{th:3.17}
Let $A$ be a 4-dimensional 3-Lie algebra with a basis
$\{e_1,e_2,e_3, e_4\}$ in Case~(2), Case (5) or Case~(6) in
Proposition~\ref{th:2.4}.  Let $r=\sum\limits_{i,j}
a^{ij}e_i\otimes e_j\in A\otimes A$ be skew-symmetric.
\begin{enumerate}
\item If $A$ is the complex 3-Lie algebra of Case~(2), then  $r$ satisfies $[[r,r,r]]=0$ if and only if
$$a^{23}(a^{12}a^{34}-a^{13}a^{24}-a^{32}a^{14})=0.$$
\item If $A$ is the complex 3-Lie algebra of Case~(5), then  $r$ satisfies $[[r,r,r]]=0$ if and only if
$$a^{34}(a^{12}a^{34}-a^{32}a^{14}+a^{42}a^{13})=0.$$
\item If $A$ is the complex 3-Lie algebra of Case~(6), then  $r$ satisfies $[[r,r,r]]=0$ if and only if
$$a^{34}(a^{12}a^{43}+a^{32}a^{14}-a^{42}a^{13})=0.$$
\end{enumerate}
\end{thm}
    \begin{proof}
        (a) For the 3-Lie algebra of Case~(2), only $T^1_{123}\neq0$. So
        \begin{align*}
            [[r,r,r]]=\sum_{p<q<r}D^{123}_{pqr}e_1\wedge e_p\wedge e_q\wedge e_r
                     =D^{123}_{234}e_1\wedge e_2\wedge e_3\wedge e_4.
        \end{align*}
        Hence $[[r,r,r]]=0$ if and only if $D^{123}_{234}=a^{23}(a^{12}a^{34}-a^{13}a^{24}-a^{32}a^{14})=0$.
        (b)For the 3-Lie algebra of Case~(5),only $T_{134}^1\neq0$ and $T_{234}^2\neq0$. So
    \begin{align*}
        [[r,r,r]]&=\sum_{p<q<r}D^{134}_{pqr}e_1\wedge e_p\wedge e_q\wedge e_r+\sum_{p<q<r}D^{234}_{pqr}e_2\wedge e_p\wedge e_q\wedge e_r\\
            &=D^{134}_{234}e_1\wedge e_2\wedge e_3\wedge e_4+D^{234}_{134}e_2\wedge e_1\wedge e_3\wedge e_4\\
            &=(D^{134}_{234}-D^{234}_{134})e_1\wedge e_2\wedge e_3\wedge e_4,
    \end{align*}
        By \text{Lemma~\ref{le:3.2}} (b), $D^{134}_{234}-D^{234}_{134}=2D^{134}_{234}$. Hence $[[r,r,r]]=0$ if and only if $D^{123}_{234}=a^{34}(a^{12}a^{34}-a^{32}a^{14}+a^{42}a^{13})=0$.\\
        (c) For the 3-Lie algebra of Case~(6), only $T^1_{234}\neq0$, $T^2_{234}\neq0$ and $T^2_{134}\neq0$. So
        \begin{align*}
            [[r,r,r]]=&\sum_{p<q<r}D^{234}_{pqr}\,\alpha\, e_1\wedge e_p\wedge e_q\wedge e_r+\sum_{p<q<r}D^{134}_{pqr}e_2\wedge e_p\wedge e_q\wedge e_r\\
                    &+\sum_{p<q<r}D^{234}_{pqr}e_2\wedge e_p\wedge e_q\wedge e_r\\
                    =&D^{234}_{234}\,\alpha\,e_1\wedge e_2\wedge e_3\wedge e_4+D^{134}_{134}e_2\wedge e_1\wedge e_3\wedge e_4+D^{234}_{134}e_2\wedge e_1\wedge e_3\wedge e_4\\
                    =&D^{234}_{134}e_2\wedge e_1\wedge e_3\wedge e_4.
        \end{align*}
        Hence $[[r,r,r]]=0$ if and only if $D^{234}_{134}=a^{34}(a^{12}a^{43}+a^{32}a^{14}-a^{42}a^{13})=0$.
    \end{proof}

\subsection{The induced local cocycle 3-Lie bialgebras}
We now provide the local cocycle 3-Lie bialgebras induced from skew-symmetric solutions of the 3-Lie CYBE.

\begin{thm} \label{th:cop}
Let $A$ be a 3-Lie algebra with a
basis $\{ e_1,\cdots,e_n\}$. Let
$r=\sum\limits_{i,j}a^{ij}e_i\otimes e_j\in A\otimes A$.
Set  $\Delta=\Delta_1+\Delta_2+\Delta_3:
A\rightarrow A\otimes A\otimes A$, in which
$\Delta_1,\Delta_2,\Delta_3$ are induced by $r$ as in Eq.~\eqref{eq:delta123}.  Then
    \begin{equation}\label{delta1}
        \Delta(x)=\sum_{i<j}\sum_{p<q}(a^{ip}a^{jq}-a^{jp}a^{iq})[x,e_i,e_j]\wedge e_p\wedge e_q, \forall x\in A.
    \end{equation}
\end{thm}
    \begin{proof}
        By Eq.~(\ref{eq:delta123}), for any $x\in A$, we have
\begin{align*}
    \Delta_1(x)=&\sum_{i,j,p,q}a^{ip}a^{jq}[x,e_i,e_j]\otimes e_p\otimes e_q
        =\sum_{i\neq j}\sum_{p,q}a^{ip}a^{jq}[x,e_i,e_j]\otimes e_p\otimes e_q\\
        =&\sum_{i<j}\sum_{p,q}(a^{ip}a^{jq}[x,e_i,e_j]+a^{jp}a^{iq}[x,e_j,e_i])\otimes e_p\otimes e_q\\
        =&\sum_{i<j}\sum_{p,q}(a^{ip}a^{jq}-a^{jp}a^{iq})[x,e_i,e_j]\otimes e_p\otimes e_q
\end{align*}
Note that $a^{ip}a^{jq}-a^{jp}a^{iq}=0$ when $p=q$. Then
\begin{align*}
    \Delta_1(x)=&\sum_{i<j}\sum_{p\neq q}(a^{ip}a^{jq}-a^{jp}a^{iq})[x,e_i,e_j]\otimes e_p\otimes e_q\\
        =&\sum_{i<j}\sum_{p<q}\big((a^{ip}a^{jq}-a^{jp}a^{iq})[x,e_i,e_j]\otimes e_p\otimes e_q+(a^{iq}a^{jp}-a^{jq}a^{ip})[x,e_i,e_j]\otimes e_q\otimes e_p\big)\\
        =&\sum_{i<j}\sum_{p<q}(a^{ip}a^{jq}-a^{jp}a^{iq})[x,e_i,e_j]\otimes(e_p\otimes e_q-e_q\otimes e_p)\\
        =&\sum_{i<j}\sum_{p<q}(a^{ip}a^{jq}-a^{jp}a^{iq})[x,e_i,e_j]\otimes(e_p\wedge e_q).
\end{align*}
Due to Eq.~(\ref{eq:delta123}) again, we have
\begin{align*}
    \Delta_2(x)=\phi_{13}\phi_{12}\Delta_1(x),\;\;
    \Delta_3(x)=\phi_{12}\phi_{13}\Delta_1(x).
\end{align*}
Therefore
\begin{align*}
    \Delta(x)=&\Delta_1(x)+\Delta_2(x)+\Delta_3(x)\\
        =&\sum_{i<j}\sum_{p<q}(a^{ip}a^{jq}-a^{jp}a^{iq})(\phi_{11}+\phi_{13}\phi_{12}+\phi_{12}\phi_{13})[x,e_i,e_j]\otimes(e_p\wedge e_q)\\
        =&\sum_{i<j}\sum_{p<q}(a^{ip}a^{jq}-a^{jp}a^{iq})[x,e_i,e_j]\wedge e_p\wedge e_q.
\end{align*}
Hence the conclusion holds.
    \end{proof}
Combining Theorem~\ref{thm:ybe}, Theorem~\ref{th:cop} and the results in the previous subsection, we obtain the following conclusion on local cocycle 3-Lie bialgebras.

\begin{pro}
Let $A$ be a 3-Lie algebra with a
basis $\{ e_1,\cdots,e_n\}$. For
$r=\sum\limits_{i,j}a^{ij}e_i\otimes e_j\in A\otimes A$, denote
$$D^{ij}_{pq}:=a^{ip}a^{jq}-a^{jp}a^{iq},\;\;\forall i,j,p,q=1,\cdots, n.$$
Then every skew-symmetric solution of the 3-Lie CYBE in
the complex 3-Lie algebras in dimension 3 and 4 gives a local
cocycle 3-Lie bialgebra $(A,\Delta)$, where $\Delta$ is given by
the following formula.
\begin{enumerate}
\item[(1)] If $A$ is the 3-dimensional 3-Lie algebra in Proposition~\ref{th:2.3}, then
\begin{align*}
  \Delta(e_1)=D^{23}_{23}e_1\wedge e_2\wedge e_3,\;\;
  \Delta(e_2)=-D^{13}_{23}e_1\wedge e_2\wedge e_3,\;\;
  \Delta(e_3)=D^{12}_{23}e_1\wedge e_2\wedge e_3.
\end{align*}
\item[(2)] If $A$ is the 4-dimensional 3-Lie algebra of Case (1) in Proposition~\ref{th:2.4}, then
\begin{align*}
    \Delta(e_1)=&(D^{34}_{34}-D^{24}_{24}+D^{23}_{23})e_2\wedge e_3\wedge e_4+(D^{24}_{12}-D^{24}_{13})e_1\wedge e_2\wedge e_3\\
                &+(D^{23}_{12}-D^{34}_{14})e_1\wedge e_2\wedge e_4+(D^{23}_{13}-D^{24}_{14})e_1\wedge e_3\wedge e_4,\\
    \Delta(e_2)=&(D^{34}_{34}-D^{13}_{13}+D^{14}_{14})e_1\wedge e_3\wedge e_4+(D^{34}_{23}-D^{14}_{12})e_1\wedge e_2\wedge e_3\\
                &+(D^{34}_{24}-D^{13}_{12})e_1\wedge e_2\wedge e_4+(D^{14}_{24}-D^{13}_{23})e_2\wedge e_3\wedge e_4,\\
    \Delta(e_3)=&(D^{12}_{12}-D^{24}_{24}+D^{14}_{14})e_1\wedge e_2\wedge e_4+(D^{14}_{13}-D^{24}_{23})e_1\wedge e_2\wedge e_3\\
                &+(D^{12}_{13}-D^{24}_{34})e_1\wedge e_3\wedge e_4+(D^{12}_{23}-D^{14}_{34})e_2\wedge e_3\wedge e_4,\\
    \Delta(e_4)=&(D^{12}_{12}-D^{13}_{13}+D^{23}_{23})e_1\wedge e_2\wedge e_3+(D^{23}_{24}-D^{13}_{14})e_1\wedge e_2\wedge e_4\\
                &+(D^{23}_{34}-D^{12}_{14})e_1\wedge e_3\wedge e_4+(D^{13}_{34}-D^{12}_{24})e_2\wedge e_3\wedge e_4.
  \end{align*}
\item[(3)] If $A$ is the 4-dimensional 3-Lie algebra of Case~(2) in Proposition~\ref{th:2.4}, then
\begin{align*}
    \Delta(e_1)=&D^{23}_{23}e_1\wedge e_2\wedge e_3+D^{23}_{24}e_1\wedge e_2\wedge e_4+D^{23}_{34}e_1\wedge e_3\wedge e_4,\\
    \Delta(e_2)=&-D^{13}_{23}e_1\wedge e_2\wedge e_3-D^{13}_{24}e_1\wedge e_2\wedge e_4-D^{13}_{34}e_1\wedge e_3\wedge e_4,\\
    \Delta(e_3)=&D^{12}_{23}e_1\wedge e_2\wedge e_3+D^{12}_{24}e_1\wedge e_2\wedge e_4+D^{12}_{34}e_1\wedge e_3\wedge e_4,\\
    \Delta(e_4)=&0,
  \end{align*}
and the parameters satisfy an additional condition $a^{23}(a^{12}a^{34}-a^{13}a^{24}-a^{32}a^{14})=0$.
\item[(4)] If $A$ is the 4-dimensional 3-Lie algebra of Case~(3) in Proposition~\ref{th:2.4}, then
\begin{align*}
    \Delta(e_1)=&0,\\
    \Delta(e_2)=&D^{34}_{23}e_1\wedge e_2\wedge e_3+D^{34}_{24}e_1\wedge e_2\wedge e_4+D^{34}_{34}e_1\wedge e_3\wedge e_4,\\
    \Delta(e_3)=&-D^{24}_{23}e_1\wedge e_2\wedge e_3-D^{24}_{24}e_1\wedge e_2\wedge e_4-D^{24}_{34}e_1\wedge e_3\wedge e_4,\\
    \Delta(e_4)=&D^{23}_{23}e_1\wedge e_2\wedge e_3+D^{23}_{24}e_1\wedge e_2\wedge e_4+D^{23}_{34}e_1\wedge e_3\wedge e_4.
  \end{align*}
\item[(5)] If $A$ is the 4-dimensional 3-Lie algebra of Case~(4) in Proposition~\ref{th:2.4}, then
  \begin{align*}
    \Delta(e_1)=&-D^{34}_{13}e_1\wedge e_2\wedge e_3-D^{34}_{14}e_1\wedge e_2\wedge e_4+D^{34}_{34}e_2\wedge e_3\wedge e_4,\\
    \Delta(e_2)=&D^{34}_{23}e_1\wedge e_2\wedge e_3+D^{34}_{24}e_1\wedge e_2\wedge e_4+D^{34}_{34}e_1\wedge e_3\wedge e_4,\\
    \Delta(e_3)=&(D^{14}_{13}-D^{24}_{23})e_1\wedge e_2\wedge e_3+(D^{14}_{14}-D^{24}_{24})e_1\wedge e_2\wedge e_4\\
                &-D^{24}_{34}e_1\wedge e_3\wedge e_4-D^{14}_{34}e_2\wedge e_3\wedge e_4,\\
    \Delta(e_4)=&(D^{23}_{23}-D^{13}_{13})e_1\wedge e_2\wedge e_3+(D^{23}_{24}-D^{13}_{14})e_1\wedge e_2\wedge e_4\\
                &+D^{23}_{34}e_1\wedge e_3\wedge e_4-D^{13}_{34}e_2\wedge e_3\wedge e_4.
  \end{align*}
Here the parameters satisfy the condition $a^{34}(a^{12}a^{34}-a^{32}a^{14}+a^{42}a^{13})=0$.
 \item[(6)] If $A$ is the 4-dimensional 3-Lie algebra of Case~(5) in Proposition~\ref{th:2.4}, then
   \begin{align*}
   \Delta(e_1)=&D^{34}_{23}e_1\wedge e_2\wedge e_3+D^{34}_{24}e_1\wedge e_2\wedge e_4+D^{34}_{34}e_1\wedge e_3\wedge e_4,\\
   \Delta(e_2)=&-D^{34}_{13}e_1\wedge e_2\wedge e_3-D^{34}_{14}e_1\wedge e_2\wedge e_4+D^{34}_{34}e_2\wedge e_3\wedge e_4,\\
   \Delta(e_3)=&(D^{24}_{13}-D^{14}_{23})e_1\wedge e_2\wedge e_3+(D^{24}_{14}-D^{14}_{24})e_1\wedge e_2\wedge e_4\\
               &-D^{24}_{34}e_1\wedge e_3\wedge e_4-D^{14}_{34}e_2\wedge e_3\wedge e_4.\\
   \Delta(e_4)=&(D^{13}_{23}-D^{23}_{13})e_1\wedge e_2\wedge e_3+(D^{13}_{24}-D^{23}_{14})e_1\wedge e_2\wedge e_4\\
               &+D^{13}_{34}e_1\wedge e_3\wedge e_4+D^{23}_{34}e_2\wedge e_3\wedge e_4.
 \end{align*}
 \item[(7)] If $A$ is the 4-dimensional 3-Lie algebra of Case~(6) in Proposition~\ref{th:2.4}, then
\begin{align*}
    \Delta(e_1)=&-D^{34}_{13}e_1\wedge e_2\wedge e_3-D^{34}_{14}e_1\wedge e_2\wedge e_4+D^{34}_{34}e_2\wedge e_3\wedge e_4,\\
    \Delta(e_2)=&(\alpha D^{34}_{23}-D^{34}_{13})e_1\wedge e_2\wedge e_3+(\alpha D^{34}_{24}-D^{34}_{14})e_1\wedge e_2\wedge e_4\\
                &+\alpha D^{34}_{34}e_2\wedge e_3\wedge e_4+D^{34}_{34}e_2\wedge e_3\wedge e_4,\\
    \Delta(e_3)=&(D^{24}_{13}+D^{14}_{13}-\alpha D^{24}_{23})e_1\wedge e_2\wedge e_3+(D^{24}_{14}+D^{14}_{14}-\alpha D^{24}_{24})e_1\wedge e_2\wedge e_4\\
                &+\alpha D^{24}_{34}e_1\wedge e_3\wedge e_4-(D^{24}_{34}+D^{14}_{34})e_2\wedge e_3\wedge e_4,\\
    \Delta(e_4)=&(\alpha D^{23}_{23}-D^{23}_{13}-D^{13}_{13})e_1\wedge e_2\wedge e_3+(\alpha D^{23}_{24}-D^{23}_{14}-D^{12}_{14})e_1\wedge e_2\wedge e_4\\
                &+\alpha D^{23}_{34}e_1\wedge e_3\wedge e_4+(D^{23}_{34}+D^{13}_{34})e_2\wedge e_3\wedge e_4,
  \end{align*}
and the parameters satisfy the condition $a^{34}(a^{12}a^{43}+a^{32}a^{14}-a^{42}a^{13})=0$.
\item[(8)] If $A$ is the 4-dimensional 3-Lie algebra of Case~(7) in Proposition~\ref{th:2.4}, then
\begin{align*}
  \Delta(e_1)=&(D^{24}_{12}-D^{24}_{13})e_1\wedge e_2\wedge e_3+D^{34}_{41}e_1\wedge e_2\wedge e_4\\
              &-D^{24}_{14}e_1\wedge e_3\wedge e_4+(D^{34}_{34}-D^{24}_{24})e_2\wedge e_3\wedge e_4,\\
  \Delta(e_2)=&(D^{34}_{23}-D^{14}_{12})e_1\wedge e_2\wedge e_3-D^{34}_{24}e_1\wedge e_2\wedge e_4\\
              &+(D^{34}_{34}-D^{14}_{14})e_1\wedge e_3\wedge e_4+D^{14}_{24}e_2\wedge e_3\wedge e_4,\\
  \Delta(e_3)=&(D^{14}_{13}-D^{24}_{23})e_1\wedge e_2\wedge e_3+(D^{14}_{14}-D^{24}_{24})e_1\wedge e_2\wedge e_4\\
              &-D^{24}_{34}e_1\wedge e_3\wedge e_4-D^{14}_{34}e_2\wedge e_3\wedge e_4,\\
  \Delta(e_4)=&(D^{12}_{12}-D^{13}_{13}+D^{23}_{23})e_1\wedge e_2\wedge e_3+(D^{23}_{24}-D^{13}_{14})e_1\wedge e_2\wedge e_4\\
              &+(D^{23}_{34}-D^{12}_{14})e_1\wedge e_3\wedge e_4+(D^{13}_{34}-D^{12}_{24})e_2\wedge e_3\wedge e_4.
\end{align*}
\end{enumerate}
\end{pro}

\section{Double construction 3-Lie bialgebras and Manin triples}
\label{sec:dc}
In this section we classify double construction
3-Lie bialgebras for complex 3-Lie algebras in dimensions
3 and 4. We also give the corresponding Manin triples.

\subsection{The double construction 3-Lie bialgebras for complex 3-Lie algebras in dimensions 3 and 4}

\begin{pro}\label{pro:4.1}
Let $A$ be a 3-Lie algebra with a basis $\{e_1,\cdots, e_n\}$. Let $\Delta:A\rightarrow A\otimes A\otimes A$ be a linear map.  Set
$$[e_a,e_b,e_c]=\sum_k T_{abc}^{k}e_k,\;\; \Delta(e_i)=\sum_{p,q,r}C_i^{pqr}e_p\otimes e_q\otimes e_r,\;\;\forall a,b,c, i=1,\cdots, n.$$
\begin{enumerate}
\item[(1)] $\Delta$ satisfies Eq.~(\ref{eq:2.6}) if and only if the following equation holds:
\begin{equation}\label{eq:constant1}
\sum_i T_{abc}^{i}C_i^{pqr}=\sum_i \big(T_{bci}^{r}C_a^{pqi}+T_{cai}^{r}C_b^{pqi}+T_{abi}^{r}C_c^{pqi}\big),\;\;\forall p,q,r, a, b, c=1,\cdots, n.
\end{equation}
\item[(2)] $\Delta$ satisfies Eq.~(\ref{eq:2.7}) if and only if the following equation holds:
\begin{equation}\label{eq:constant2}
\sum_i T_{abc}^{i}C_i^{pqr}=\sum_i \big(T_{bci}^{r}C_a^{pqi}+T_{bci}^{q}C_a^{pir}+T_{bci}^{p}C_a^{iqr}\big), \;\;\forall  p,q,r, a, b, c=1,\cdots, n.
\end{equation}
\end{enumerate}
\end{pro}

\begin{proof}
It is obtained by a straightforward computation of Eqs.~(\ref{eq:2.6}) and (\ref{eq:2.7}) followed by comparing the coefficients.
\end{proof}

With the conditions and notation as above, let $\{f_1,f_2,\cdots,f_n\}$ be the dual basis of $A^{\ast}$ and $\Delta^*:A^*\otimes A^*\otimes A^*\rightarrow A^*$ be the dual map. Then a direct computation shows that
\begin{equation}\label{def.c}
\Delta^{\ast}(f_p\otimes f_q\otimes f_r)=\sum_i C_i^{pqr}f_i,\;\;\forall p, q, r=1,\cdots, n.
\end{equation}
If in addition, $\Delta^{\ast}$ defines a 3-Lie algebra structure on $A^{\ast}$, then $\Delta$ is a skew-symmetric linear map, i.e., for any permutation $\ttau$ on $\{p,q,r\}$,
\begin{equation}\label{c}
C_i^{\ttau(p)\ttau(q)\ttau(r)} =\textrm{sgn}(\ttau)C_i^{pqr},\;\;\forall p,q,r=1,\cdots,n.
\end{equation}
\begin{lem}\label{le:4.1}
With the notations as above. If Eq.~(\ref{c}) holds, then $\Delta$ satisfies Eq.~(\ref{eq:2.6}) if and only if Eq.~(\ref{eq:constant1}) holds for any $p,q,r$ and $a<b<c$.
\end{lem}
    \begin{proof}
       At first, we claim that the following two equations are equivalent when $p,q,r$ are fixed.
        \begin{align*}
            \sum_i T_{a_1b_1c_1}^{i}C_i^{pqr}&=\sum_i \big(T_{b_1c_1i}^{r}C_{a_1}^{pqi}+T_{c_1a_1i}^{r}C_{b_1}^{pqi}+T_{a_1b_1i}^{r}C_{c_1}^{pqi}\big),\\
            \sum_i T_{a_2b_2c_2}^{i}C_i^{pqr}&=\sum_i \big(T_{b_2c_2i}^{r}C_{a_2}^{pqi}+T_{c_2a_2i}^{r}C_{b_2}^{pqi}+T_{a_2b_2i}^{r}C_{c_2}^{pqi}\big),
        \end{align*}
        where $a_2,b_2,c_2$ are obtained by permuting $a_1,b_1,c_1$. In fact, without loss of generality, we assume $a_1=b_2,b_1=a_2,c_1=c_2$. Then
        \begin{align*}
            &\sum_i T_{a_1b_1c_1}^{i}C_i^{pqr}=\sum_i \big(T_{b_1c_1i}^{r}C_{a_1}^{pqi}+T_{c_1a_1i}^{r}C_{b_1}^{pqi}+T_{a_1b_1i}^{r}C_{c_1}^{pqi}\big)\\
            \Leftrightarrow &\sum_i T_{b_2a_2c_2}^{i}C_i^{pqr}=\sum_i \big(T_{a_2c_2i}^{r}C_{b_2}^{pqi}+T_{c_2b_2i}^{r}C_{a_2}^{pqi}+T_{b_2a_2i}^{r}C_{c_2}^{pqi}\big)\\
            \Leftrightarrow &\sum_i (-T_{a_2b_2c_2}^{i})C_i^{pqr}=\sum_i \big((-T_{c_2a_2i}^{r})C_{b_2}^{pqi}+(-T_{b_2c_2i}^{r})C_{a_2}^{pqi}+(-T_{a_2b_2i}^{r})C_{c_2}^{pqi}\big)\\
            \Leftrightarrow &\sum_i T_{a_2b_2c_2}^{i}C_i^{pqr}=\sum_i \big(T_{c_2a_2i}^{r}C_{b_2}^{pqi}+T_{b_2c_2i}^{r}C_{a_2}^{pqi}+T_{a_2b_2i}^{r}C_{c_2}^{pqi}\big).
        \end{align*}
        Furthermore, if any two of $a,b,c$ are equal, then Eq.~(\ref{eq:constant1}) holds automatically. In fact, assume $a=b$ without loss of generality. The left hand side of Eq.~(\ref{eq:constant1}) is zero because $T^i_{aac}=0$, whereas the right hand side is also zero because $T_{aci}^{r}C_a^{pqi}=-T_{aci}^{q}C_a^{pir}$ and $T_{aai}^{r}=0$.  Therefore the indices $a,b,c$ should be selected distinct and the sequence of $a,b,c$ makes no difference. Hence the lemma holds.
    \end{proof}

\begin{thm}\label{th:4.1}
Let $A$ be the 3-dimensional 3-Lie algebra given in
Proposition~\ref{th:2.3}. If a skew-symmetric linear map
$\Delta:A\rightarrow A\otimes A\otimes A$  satisfies
Eq.~(\ref{eq:2.6}) or Eq.~(\ref{eq:constant1}),  then $\Delta=0$.
Therefore there is no non-trivial double construction 3-Lie bialgebra for $A$.
\end{thm}
    \begin{proof}
     With the notations as in Proposition~\ref{pro:4.1}. Fix $p,q,r$. Then by lemma~\ref{le:4.1}, we only need to consider the following three equations.
        \begin{align*}
            \sum_i T_{123}^{i}C_i^{pq1}&=\sum_i \big(T_{23i}^{1}C_{1}^{pqi}+T_{31i}^{1}C_{2}^{pqi}+T_{12i}^{1}C_{3}^{pqi}\big),\\
            \sum_i T_{123}^{i}C_i^{pq2}&=\sum_i \big(T_{23i}^{2}C_{1}^{pqi}+T_{31i}^{2}C_{2}^{pqi}+T_{12i}^{2}C_{3}^{pqi}\big),\\
            \sum_i T_{123}^{i}C_i^{pq3}&=\sum_i \big(T_{23i}^{3}C_{1}^{pqi}+T_{31i}^{3}C_{2}^{pqi}+T_{12i}^{3}C_{3}^{pqi}\big).
        \end{align*}
        \quad Simplifying those equations, we have
        \begin{align*}
            C_2^{pq2}+C_3^{pq3}=0,\;\; C_1^{pq2}=0,\;\; C_1^{pq3}=0.
        \end{align*}
        Since $p,q,r$ are chosen arbitrarily and $\Delta$ is skew-symmetric, this shows that $C_i^{abc}=0$, $\forall i,a,b,c=1,2,3$, i.e., $\Delta=0$.
    \end{proof}

\begin{lem}\label{th:3}  For any 3-Lie algebra in dimension 4, as displayed in Proposition~\ref{th:2.4}, the skew-symmetric linear map $\Delta:A\rightarrow A\otimes A\otimes
A$ satisfying Eq.~(\ref{eq:constant1}) is given as
follows (all the parameters are arbitrary constants).
\begin{enumerate}
\item[(1)] If $A$ is the 4-dimensional 3-Lie algebra of Case~(1) in Proposition~\ref{th:2.4}, then
\begin{equation}\label{re:1}
  \begin{cases}
    &\Delta(e_1)=ke_2\wedge e_3\wedge e_4,\\
    &\Delta(e_2)=ke_1\wedge e_3\wedge e_4,\\
    &\Delta(e_3)=ke_1\wedge e_2\wedge e_4,\\
    &\Delta(e_4)=ke_1\wedge e_2\wedge e_3.
  \end{cases}
\end{equation}
\item[(2)] If $A$ is the 4-dimensional 3-Lie algebra of Case~(2) in Proposition~\ref{th:2.4}, then
\begin{equation}\label{re:2}
\begin{cases}
\Delta(e_1)=\Delta(e_4)=0,\\
\Delta(e_2)=ke_1\wedge e_2\wedge e_4+c_1e_1\wedge e_3\wedge e_4,\\
\Delta(e_3)=-ke_1\wedge e_3\wedge e_4+c_2e_1\wedge e_2\wedge e_4.
\end{cases}
\end{equation}
\item[(3)] If $A$ is the 4-dimensional 3-Lie algebra of Case~(3) in Proposition~\ref{th:2.4}, then
\begin{equation}\label{re:3}
\begin{cases}
\Delta(e_1)=0,\\
\Delta(e_2)=k_1e_1\wedge e_2\wedge e_3+k_2e_1\wedge e_2 \wedge e_4+c_1e_1\wedge e_3 \wedge e_4,\\
\Delta(e_3)=k_3e_1\wedge e_2\wedge e_3-k_2e_1\wedge e_3 \wedge e_4+c_2e_1\wedge e_2 \wedge e_4,\\
\Delta(e_4)=-k_3e_1\wedge e_2\wedge e_4+k_1e_1\wedge e_3 \wedge e_4+c_3e_1\wedge e_2\wedge e_3.
\end{cases}
\end{equation}
\item[(4)] If $A$ is the 4-dimensional 3-Lie algebra of Case~(4) in Proposition~\ref{th:2.4}, then
\begin{equation}\label{re:4}
\begin{cases}
\Delta(e_1)=\Delta(e_2)=0,\\
\Delta(e_3)=ke_1\wedge e_2\wedge e_3+c_1e_1\wedge e_2\wedge e_4,\\
\Delta(e_4)=-ke_1\wedge e_3\wedge e_4+c_2e_1\wedge e_2\wedge e_3.
\end{cases}
\end{equation}
\item[(5)] If $A$ is the 4-dimensional 3-Lie algebra of Case~(5) in Proposition~\ref{th:2.4}, then
\begin{equation}\label{re:5}
\begin{cases}
\Delta(e_1)=\Delta(e_2)=0,\\
\Delta(e_3)=ke_1\wedge e_2\wedge e_3+c_1e_1\wedge e_2\wedge e_4,\\
\Delta(e_4)=-ke_1\wedge e_2\wedge e_4+c_2e_1\wedge e_2\wedge e_3.
\end{cases}
\end{equation}
\item[(6)] If $A$ is the 4-dimensional 3-Lie algebra of Case~(6) in Proposition~\ref{th:2.4}, then
\begin{equation}\label{re:6}
\begin{cases}
\Delta(e_1)=\Delta(e_2)=0,\\
\Delta(e_3)=ke_1\wedge e_2\wedge e_3+c_1e_1\wedge e_2\wedge e_4,\\
\Delta(e_4)=-ke_1\wedge e_2\wedge e_4+c_2e_1\wedge e_2\wedge e_3.
\end{cases}
\end{equation}
\item[(7)] If $A$ is the 4-dimensional 3-Lie algebra of Case~(7) in Proposition~\ref{th:2.4}, then
\begin{equation}\label{re:7}
\begin{cases}
\Delta(e_1)=\Delta(e_2)=\Delta(e_3)=0,\\
\Delta(e_4)=ce_1\wedge e_2\wedge e_3.
\end{cases}
\end{equation}
\end{enumerate}
\end{lem}
    \begin{proof}
        We give an explicit proof for the Case~(1) as an example and we omit the proofs for the other cases since the proofs are similar.
        Fix $p,q,r$. Then by Lemma~\ref{le:4.1}, we only need to consider Eq.~(\ref{eq:constant1}) whose indices $(r,a,b,c)$ are given by the following quadruples.
        \begin{equation}
        (i,1,2,3),~~~ (i,1,2,4), ~~~ (i,1,3,4),~~~ (i,2,3,4), ~~~ 1\leq i\leq 4.
        \label{eq:rabc}
        \end{equation}

    Let $A$ be the 3-Lie algebra of Case~(1). For the quadruples $(i,1,2,3), 1\leq i\leq 4$, we have
        \begin{align}
            (r,a,b,c)=(1,1,2,3):\;\;\;&C^{pq1}_4=C^{pq4}_1,\label{delta1.1}\\
            (r,a,b,c)=(2,1,2,3):\;\;\;&C^{pq2}_4=-C^{pq4}_2,\label{delta1.2}\\
            (r,a,b,c)=(3,1,2,3):\;\;\;&C^{pq3}_4=C^{pq4}_3,\label{delta1.3}\\
            (r,a,b,c)=(4,1,2,3):\;\;\;&C^{pq4}_4=C^{pq1}_1+C^{pq2}_2+C^{pq3}_3.\label{delta1.4}
        \end{align}
 By Eq.~(\ref{delta1.1}), we obtain
 $$C^{124}_4=C^{241}_4=C^{244}_1=0\;\;{\rm and}\;\; C^{134}_4=C^{341}_4=C^{344}_1=0.$$
 By Eq.~(\ref{delta1.2}), we obtain $$C^{234}_4=C^{342}_4=C^{344}_2=0.$$
 Hence $C^{pq4}_4=0$, $\forall p,q=1,2,3,4$.

 Similarly, for the other rows of equations, we show that
 $$C^{pq3}_3=0,\;\;C^{pq2}_2=0,\;\; C^{pq1}_1=0,\;\; \forall p,q=1,2,3,4.$$
 That is, $C^{pqr}_i=0$ if any one of $p,q,r$ equals $i$. What remain unknown in $\{C^{abc}_i|a<b<c\}$ are $C^{123}_4$, $C^{124}_3$, $C^{134}_2$ and $C^{234}_1$. By Eq.~(\ref{delta1.1}) again, we obtain
 $$C^{123}_4=C^{231}_4=C^{234}_1.$$
 By Eq.~(\ref{delta1.2}), we obtain
 $$C^{123}_4=-C^{132}_4=C^{134}_2.$$
 By Eq.~(\ref{delta1.3}), we obtain
 $$C^{123}_4=C^{124}_3.$$ Therefore
        \begin{equation}\label{C}
            C^{123}_4=C^{124}_3=C^{134}_2=C^{234}_1.
        \end{equation}
Furthermore, it is straightforward to check that Eq.~(\ref{C}) satisfies all the 16 equations in Eq.~(\ref{eq:rabc}).  Therefore  $\Delta$ is determined explicitly in Eq.~(\ref{re:1}) by taking $C^{123}_4=k$.
    \end{proof}

\begin{lem}\label{delta2}
    Let $A$ be a complex 3-Lie algebra with a basis $\{e_1,e_2,\cdots,e_n\}$. Let $p,q,r,s,t$ be fixed indexes, and $m_1,m_2,m_3,m_4\in \mathbb C$. Assume that
    $$\mathrm{ad}_{e_s,e_t}(e_p)=m_1e_p+m_4e_q,~\mathrm{ad}_{e_s,e_t}(e_q)=m_2e_q,~\mathrm{ad}_{e_s,e_t}(e_r)=m_3e_r,$$
    and set
    $$\Phi_{e_s,e_t}=\id\otimes \id \otimes \mathrm{ad}_{e_s,e_t}+\id\otimes \mathrm{ad}_{e_s,e_t}\otimes \id +\mathrm{ad}_{e_s,e_t}\otimes \id \otimes \id .$$
    Then
    \begin{equation}\label{Phi}
        \Phi_{e_s,e_t}(e_p\wedge e_q\wedge e_r)=(m_1+m_2+m_3)e_p\wedge e_q\wedge e_r.
    \end{equation}
\end{lem}
    \begin{proof}
        Assume $m_2=m_3=m_4=0$. Then
        \begin{align}
            \id\otimes \id \otimes \mathrm{ad}_{e_s,e_t}(e_p\wedge e_q\wedge e_r)&=\id\otimes \id \otimes \mathrm{ad}_{e_s,e_t}(e_q\otimes e_r\otimes e_p-e_r\otimes e_q\otimes e_p)\nonumber\\
                &=m_1(e_q\otimes e_r\otimes e_p-e_r\otimes e_q\otimes e_p),\label{Phi1}\\
            \id\otimes \mathrm{ad}_{e_s,e_t}\otimes \id (e_p\wedge e_q\wedge e_r)&=\id\otimes \mathrm{ad}_{e_s,e_t}\otimes \id (e_r\otimes e_p\otimes e_q-e_q\otimes e_p\otimes e_r)\nonumber\\
                &=m_1(e_r\otimes e_p\otimes e_q-e_q\otimes e_p\otimes e_r),\label{Phi2}\\
            \mathrm{ad}_{e_s,e_t}\otimes \id \otimes \id (e_p\wedge e_q\wedge e_r)&=\mathrm{ad}_{e_s,e_t}\otimes \id \otimes \id (e_p\otimes e_q\otimes e_r-e_p\otimes e_r\otimes e_q)\nonumber\\
                &=m_1(e_p\otimes e_q\otimes e_r-e_p\otimes e_r\otimes e_q).\label{Phi3}
        \end{align}
 Adding Eqs.~(\ref{Phi1}) -- (\ref{Phi3}) together, we have
        \begin{equation}\label{Phim1}
            \Phi_{e_s,e_t}(e_p\wedge e_q\wedge e_r)=m_1e_p\wedge e_q\wedge e_r.
        \end{equation}
        Assume $m_1=m_3=m_4=0$. Then
        \begin{equation}\label{Phim2}
            \Phi_{e_s,e_t}(e_p\wedge e_q\wedge e_r)=\Phi_{e_s,e_t}(-e_q\wedge e_p\wedge e_r)=-m_2e_q\wedge e_p\wedge e_r=m_2e_p\wedge e_q\wedge e_r.
        \end{equation}
        Assume $m_1=m_2=m_4=0$. Then
        \begin{equation}\label{Phim3}
            \Phi_{e_s,e_t}(e_p\wedge e_q\wedge e_r)=\Phi_{e_s,e_t}(e_r\wedge e_p\wedge e_q)=m_3e_r\wedge e_p\wedge e_q=m_3e_p\wedge e_q\wedge e_r.
        \end{equation}
        Assume $m_1=m_2=m_3=0$. Then
        \begin{align}
            \id\otimes \id \otimes \mathrm{ad}_{e_s,e_t}(e_p\wedge e_q\wedge e_r)&=\id\otimes \id \otimes \mathrm{ad}_{e_s,e_t}(e_q\otimes e_r\otimes e_p-e_r\otimes e_q\otimes e_p)\nonumber\\
                &=m_4(e_q\otimes e_r\otimes e_q-e_r\otimes e_q\otimes e_q),\label{Phi4}\\
            \id\otimes \mathrm{ad}_{e_s,e_t}\otimes \id (e_p\wedge e_q\wedge e_r)&=\id\otimes \mathrm{ad}_{e_s,e_t}\otimes \id (e_r\otimes e_p\otimes e_q-e_q\otimes e_p\otimes e_r)\nonumber\\
                &=m_4(e_r\otimes e_q\otimes e_q-e_q\otimes e_q\otimes e_r),\label{Phi5}\\
            \mathrm{ad}_{e_s,e_t}\otimes \id \otimes \id (e_p\wedge e_q\wedge e_r)&=\mathrm{ad}_{e_s,e_t}\otimes \id \otimes \id (e_p\otimes e_q\otimes e_r-e_p\otimes e_r\otimes e_q)\nonumber\\
                &=m_4(e_q\otimes e_q\otimes e_r-e_q\otimes e_r\otimes e_q).\label{Phi6}
        \end{align}
Adding Eqs.~(\ref{Phi4}) -- (\ref{Phi6}) together, we have
        \begin{equation}\label{Phim4}
            \Phi_{e_s,e_t}(e_p\wedge e_q\wedge e_r)=0.
        \end{equation}
        Since $\Phi_{e_s,e_t}$ is linear, Eqs.~(\ref{Phim1}), (\ref{Phim2}), (\ref{Phim3}) and (\ref{Phim4}) together indicate that Eq.~(\ref{Phi}) holds.
    \end{proof}

\begin{thm}
Let $A$ be one of the 4-dimensional 3-Lie algebras of Cases~(2), (5) and (6) given in Proposition~\ref{th:2.4}. Then any  double construction 3-Lie bialgebra for $A$ is trivial.
\end{thm}

    \begin{proof}
        By Lemma~\ref{th:3}, we need to show that for the mentioned cases, if in addition $\Delta$ satisfies Eq.~(\ref{eq:2.7}), then $\Delta=0$.

        Case (2): Substituting $x=e_2,y=e_2,z=e_3$ into Eq.~(\ref{eq:2.7}), we get
        \begin{align*}
            0=\Phi_{e_2,e_3}(\Delta(e_2))=&\Phi_{e_2,e_3}(ke_1\wedge e_2\wedge e_4+c_1e_1\wedge e_3\wedge e_4)
                =ke_1\wedge e_2\wedge e_4+c_1e_1\wedge e_3\wedge e_4.
        \end{align*}
Hence $k=c_1=0$.  Substituting $x=e_3,y=e_2,z=e_3$ into Eq.~(\ref{eq:2.7}), we get
       \begin{align*}
            0=\Phi_{e_2,e_3}(\Delta(e_3))
                =\Phi_{e_2,e_3}(-ke_1\wedge e_3\wedge e_4+c_2e_1\wedge e_2\wedge e_4)
                =-ke_1\wedge e_3\wedge e_4+c_2e_1\wedge e_2\wedge e_4.
       \end{align*}
Hence $c_2=0$. Therefore $\Delta=0$.

       Case (5): Substituting $x=e_3,y=e_3,z=e_4$ into Eq.~(\ref{eq:2.7}), we get
       \begin{align*}
            0=\Phi_{e_3,e_4}(\Delta(e_3))
                =\Phi_{e_3,e_4}(ke_1\wedge e_2\wedge e_3+c_1e_1\wedge e_2\wedge e_4)
                =2ke_1\wedge e_2\wedge e_3+2c_1e_1\wedge e_2\wedge e_4.
       \end{align*}
 Hence $k=c_1=0$.   Substituting $x=e_4,y=e_3,z=e_4$ into Eq.~(\ref{eq:2.7}), we get
       \begin{align*}
            0=\Phi_{e_3,e_4}(\Delta(e_4))
                =\Phi_{e_3,e_4}(-ke_1\wedge e_2\wedge e_4+c_2e_1\wedge e_2\wedge e_3)
                =-2ke_1\wedge e_2\wedge e_4+2c_2e_1\wedge e_2\wedge e_3.
       \end{align*}
Hence $c_2=0$. Therefore $\Delta=0$.

       Case (6): Substituting $x=e_3,y=e_3,z=e_4$ into Eq.~(\ref{eq:2.7}), we get
       \begin{align*}
            0=\Phi_{e_3,e_4}(\Delta(e_3))
                =\Phi_{e_3,e_4}(ke_1\wedge e_2\wedge e_3+c_1e_1\wedge e_2\wedge e_4)
                =ke_1\wedge e_2\wedge e_3+c_1e_1\wedge e_2\wedge e_4.
       \end{align*}
       Hence $k=c_1=0$. Substituting $x=e_4,y=e_3,z=e_4$ into Eq.~(\ref{eq:2.7}), we get
       \begin{align*}
            0=\Phi_{e_3,e_4}(\Delta(e_4))
                =\Phi_{e_3,e_4}(-ke_1\wedge e_2\wedge e_4+c_2e_1\wedge e_2\wedge e_3)
                =-ke_1\wedge e_2\wedge e_4+c_2e_1\wedge e_2\wedge e_3.
       \end{align*}
       Hence $c_2=0$. Therefore $\Delta=0$.
    \end{proof}

\begin{lem}\label{C2}
     With the notations as in Proposition~\ref{pro:4.1}. If Eq.~ (\ref{c}) holds, then $\Delta$ satisfies Eq.~ (\ref{eq:2.7}) if and only if Eq.~(\ref{eq:constant2}) holds for any $a$ and $b<c,~p<q<r$.
\end{lem}
    \begin{proof}
    It follows from a proof similar to the one for Lemma~\ref{le:4.1}.
    \end{proof}

\begin{thm}\label{th:1111} Let $A$ be a 4-dimensional 3-Lie algebra with a
basis $\{ e_1,\cdots,e_4\}$.
\begin{enumerate}
\item[(1)] If $A$ is the 4-dimensional 3-Lie algebra of Case~(1)
given in Proposition~\ref{th:2.4}, then $(A,\Delta)$ is a double
construction 3-Lie bialgebra, where $\Delta$ is
given by Eq.~(\ref{re:1}).
\item[(2)] If $A$ is the 4-dimensional
3-Lie algebra of Case~(3) given in Proposition~\ref{th:2.4}, then
$(A,\Delta)$ is a double construction 3-Lie bialgebra, where $\Delta$ is given by Eq.~(\ref{re:3}).
\item[(3)] If $A$ is the 4-dimensional 3-Lie algebra of Case~(4)
given in Proposition~\ref{th:2.4}, then $(A,\Delta)$ is a double
construction 3-Lie bialgebra, where $\Delta$ is
given by Eq.~(\ref{re:4}). \item[(4)] If $A$ is the 4-dimensional
3-Lie algebra of Case~(7) given in Proposition~\ref{th:2.4}, then
$(A,\Delta)$ is a double construction 3-Lie
bialgebra, where $\Delta$ is given by Eq.~(\ref{re:7}).
\end{enumerate}
\end{thm}

\begin{proof}
In fact, for the 4-dimensional 3-Lie algebras of Cases~(1), (3), (4) and (7),
 the corresponding $\Delta$ appearing in Lemma~\ref{th:3} satisfies  Eq.~(\ref{eq:2.7}), too. We give an explicit proof for the Case~(1) as an example and we omit the proofs for the other cases since they are similar. For the 3-Lie algebra of Case~(1), fix $a,b,c$. By Lemma~\ref{C2}, we  only need to consider othe following four equations.
    \begin{align}
        &(p,q,r)=(1,2,3):&\sum_i T_{abc}^{i}C_i^{123}=\sum_i \big(T_{bci}^{3}C_a^{12i}+T_{bci}^{2}C_a^{1i3}+T_{bci}^{1}C_a^{i23}\big),\label{pqr1.1}\\
        &(p,q,r)=(1,2,4):&\sum_i T_{abc}^{i}C_i^{124}=\sum_i \big(T_{bci}^{4}C_a^{12i}+T_{bci}^{2}C_a^{1i4}+T_{bci}^{1}C_a^{i24}\big),\label{pqr1.2}\\
        &(p,q,r)=(1,3,4):&\sum_i T_{abc}^{i}C_i^{134}=\sum_i \big(T_{bci}^{4}C_a^{13i}+T_{bci}^{3}C_a^{1i4}+T_{bci}^{1}C_a^{i34}\big),\label{pqr1.3}\\
        &(p,q,r)=(2,3,4):&\sum_i T_{abc}^{i}C_i^{234}=\sum_i \big(T_{bci}^{4}C_a^{23i}+T_{bci}^{3}C_a^{2i4}+T_{bci}^{2}C_a^{i34}\big).\label{pqr1.4}
    \end{align}
For Eq.~(\ref{pqr1.1}), the left hand side is $$\sum_iT_{abc}^iC_i^{123}=T_{abc}^4,$$
whereas the right hand side is
    \begin{align*}
        &\sum_i \big(T_{bci}^{3}C_a^{12i}+T_{bci}^{2}C_a^{1i3}+T_{bci}^{1}C_a^{i23}\big)\\
        =&T_{bc3}^{3}C_a^{123}+T_{bc4}^{3}C_a^{124}+T_{bc2}^{2}C_a^{123}+T_{bc4}^{2}C_a^{143}+T_{bc1}^{1}C_a^{123}+T_{bc4}^{1}C_a^{423}\\
        =&\big(T^3_{bc3}+T^2_{bc2}+T^1_{bc1}\big)C^{123}_a+T_{bc4}^{3}C_a^{124}-T_{bc4}^{2}C_a^{134}+T_{bc4}^{1}C_a^{234}\\
        =&0+T_{bc4}^{3}C_a^{124}-T_{bc4}^{2}C_a^{134}+T_{bc4}^{1}C_a^{234}.
    \end{align*}
Therefore,  Eq.~(\ref{pqr1.1}) holds if and only if the following series of equations hold:
    \begin{align*}
        a=1:~~&T_{1bc}^4=0+0+T_{bc4}^1,\\
        a=2:~~&T_{2bc}^4=0-T_{bc4}^2+0,\\
        a=3:~~&T_{3bc}^4=T_{bc4}^3+0+0,\\
        a=4:~~&0=0.
    \end{align*}
It is straightforward to  show that these equations hold for arbitrary $b,c=1,2,3,4$. Hence Eq.~(\ref{pqr1.1}) holds. Similarly, Eqs.~(\ref{pqr1.2}) -- (\ref{pqr1.4}) also hold. Therefore Eq.~(\ref{eq:2.7}) holds.

 Moreover, it is straightforward to check (also see the remark after this proof) that, for every $\Delta$ appearing in the conclusion, the dual $\Delta^*$ defines a 3-Lie algebra on $A^*$. Hence the conclusion holds.
 \end{proof}

\begin{rmk}{\rm We give explicitly the 3-Lie algebra structure on the dual space $A^*$ obtained from $\Delta$ in the above double construction 3-Lie bialgebras as follows.
\begin{enumerate}
\item[(1)]  $A$ is the 4-dimensional 3-Lie algebra of Case~(1) given in Proposition~\ref{th:2.4}.
\begin{equation*}
  [e^{\ast}_1,e^{\ast}_2,e^{\ast}_3]^{\ast}=ke^{\ast}_4,~
  [e^{\ast}_1,e^{\ast}_2,e^{\ast}_4]^{\ast}=ke^{\ast}_3,~
  [e^{\ast}_1,e^{\ast}_3,e^{\ast}_4]^{\ast}=ke^{\ast}_2,~
  [e^{\ast}_2,e^{\ast}_3,e^{\ast}_4]^{\ast}=ke^{\ast}_1.
\end{equation*}
\item[(2)] $A$ is the 4-dimensional 3-Lie algebra of Case~(3) given in Proposition~\ref{th:2.4}.
\begin{align*}
  ~[e^{\ast}_1,e^{\ast}_2,e^{\ast}_3]^{\ast}&=k_1e^{\ast}_2+k_3e^{\ast}_3+c_3e^{\ast}_4,\\
  ~[e^{\ast}_1,e^{\ast}_2,e^{\ast}_4]^{\ast}&=k_2e^{\ast}_2+c_2e^{\ast}_3-k_3e^{\ast}_4,\\
  ~[e^{\ast}_1,e^{\ast}_3,e^{\ast}_4]^{\ast}&=c_1e^{\ast}_2-k_2e^{\ast}_3+k_1e^{\ast}_4.
\end{align*}
\item[(3)] $A$ is the 4-dimensional 3-Lie algebra of Case~(4) given in Proposition~\ref{th:2.4}.
\begin{equation*}
  [e^{\ast}_1,e^{\ast}_2,e^{\ast}_3]^{\ast}=ke^{\ast}_3+c_2e^{\ast}_4,~
  [e^{\ast}_1,e^{\ast}_2,e^{\ast}_4]^{\ast}=c_1e^{\ast}_3-ke^{\ast}_4.
\end{equation*}
\item[(4)] $A$ is the 4-dimensional 3-Lie algebra of Case~(7) given in Proposition~\ref{th:2.4}.
\begin{equation*}
[e^{\ast}_1,e^{\ast}_2,e^{\ast}_3]^{\ast}=ce^{\ast}_4.
\end{equation*}
\end{enumerate}
It is easy to show that $(A^{\ast},[\cdot,\cdot,\cdot]^{\ast})$ in
Case~(1) for $k\ne 0$ and in Case~(7) for $c\ne 0$ are
respectively isomorphic to the 3-Lie algebras of the Case~(1) and
Case~(3)
given in Proposition~\ref{th:2.4}. For Cases~(3) and
(4) mentioned in the remark, $(A^{\ast},[\cdot,\cdot,\cdot]^{\ast})$ is still a 3-Lie algebra.}
\end{rmk}

\subsection{Pseudo-metric 3-Lie algebras in dimension 8}

By Theorem~\ref{thm:relations} and the results in the previous
subsection, we can get the corresponding pseudo-metric 3-Lie
algebras in dimension 8 (Manin triples of 3-Lie algebras) as
follows.

\begin{thm} Let $A$ be a 4-dimensional vector space with a basis $\{e_1,\cdots, e_4\}$ and $\{e^{\ast}_1,\cdots,e^{\ast}_4\}$ be the dual basis.
On the vector space $A\oplus A^*$ define a bilinear form  $(\cdot,\cdot)_{+}$ by Eq.~\eqref{eq:bf}, that is, with respect to the basis $\{e_1,\cdots,e_4, e^{\ast}_1,\cdots,e^{\ast}_4\}$, it corresponds to the matrix $
    \begin{pmatrix}
  0 & I_n \\ I_n & 0
    \end{pmatrix}$. We can get the following families of 8-dimensional pseudo-metric 3-Lie algebras $(A\oplus A^*, (\cdot,\cdot)_{+})$.
\begin{enumerate}
\item[(1)] The non-zero product of
    3-Lie algebra structure on $A\oplus A^*$ is given by
\begin{eqnarray*}
&~~[e_1,e_2,e_3]=e_4,[e_1,e_2,e_4]=e_3,[e_1,e_3,e_4]=e_2,[e_2,e_3,e_4]=e_1;\\
&~~[e^{\ast}_1,e^{\ast}_2,e^{\ast}_3]^{\ast}=ke^{\ast}_4,~[e^{\ast}_1,e^{\ast}_2,e^{\ast}_4]^{\ast}=ke^{\ast}_3,~[e^{\ast}_1,e^{\ast}_3,e^{\ast}_4]^{\ast}=ke^{\ast}_2,~ [e^{\ast}_2,e^{\ast}_3,e^{\ast}_4]^{\ast}=ke^{\ast}_1,\\
&~~[e_i,e^{\ast}_j,e^{\ast}_k]=e^{\ast}_m,~[e_i,e_j,e^{\ast}_k]=e^{\ast}_m.
\end{eqnarray*}
where the last equation holds for $i<j<k$ and $m$ which is distinct from $i,j,k$.  They correspond to the double construction bialgebras $(A,\Delta)$ given in
Theorem~\ref{th:1111}, where $A$ is the 3-Lie algebra of Case~(1) given in Proposition~\ref{th:2.4}.
\item[(2)]The non-zero product of
    3-Lie algebra structure on $A\oplus A^*$ is given by
\begin{eqnarray*}
  &~~[e_2,e_3.e_4]=e_1,\\
  &~~[e^{\ast}_1,e^{\ast}_2,e^{\ast}_3]^{\ast}=k_1e^{\ast}_2+k_3e^{\ast}_3+c_3e^{\ast}_4,\\
  &~~[e^{\ast}_1,e^{\ast}_2,e^{\ast}_4]^{\ast}=k_2e^{\ast}_2+c_2e^{\ast}_3-k_3e^{\ast}_4,\\
  &~~[e^{\ast}_1,e^{\ast}_3,e^{\ast}_4]^{\ast}=c_1e^{\ast}_2-k_2e^{\ast}_3+k_1e^{\ast}_4,\\
  &~~[e_1,e_2,e^{\ast}_1] = -e^{\ast}_3,[e_1,e_3,e^{\ast}_1] = e{\ast}_2,[e_2,e_3,e^{\ast}_1] = -e{\ast}_1,\\
  &~~[e_2,e^{\ast}_1,e^{\ast}_2]=-k_1e_3-k_2e_4,[e_2,e^{\ast}_2,e^{\ast}_3]=-k_1e_1,[e_2,e^{\ast}_1,e^{\ast}_3]=k_1e_2-c_1e_4,\\
  &~~[e_2,e^{\ast}_2,e^{\ast}_4]=-k_2e_1,[e_2,e^{\ast}_1,e^{\ast}_4]=k_2e_2+c_1e_3,[e_2,e^{\ast}_3,e^{\ast}_4]=-c_1e_1,\\
  &~~[e_3,e^{\ast}_1,e^{\ast}_2]=-k_3e_3-c_2e_4,[e_3,e^{\ast}_2,e^{\ast}_3]=-k_3e_1,[e_3,e^{\ast}_1,e^{\ast}_3]=k_3e_2+k_2e_4,\\
  &~~[e_3,e^{\ast}_2,e^{\ast}_4]=-c_2e_1,[e_3,e^{\ast}_1,e^{\ast}_4]=c_2e_2-k_2e_3,[e_3,e^{\ast}_3,e^{\ast}_4]=k_2e_1,\\
  &~~[e_4,e^{\ast}_1,e^{\ast}_2]=-c_2e_3+k_3e_4,[e_4,e^{\ast}_2,e^{\ast}_3]=-c_3e_1,[e_4,e^{\ast}_1,e^{\ast}_3]=c_3e_2-k_1e_4,\\
  &~~[e_4,e^{\ast}_2,e^{\ast}_4]=k_3e_1,[e_4,e^{\ast}_1,e^{\ast}_4]=-k_3e_2+k_1e_3,[e_4,e^{\ast}_3,e^{\ast}_4]=-k_1e_1.
\end{eqnarray*}
They correspond to the double construction bialgebras $(A,\Delta)$ given in
Theorem~\ref{th:1111}, where $A$ is the 3-Lie algebra of Case~(3) given in Proposition~\ref{th:2.4}.
\item[(3)] The non-zero product of
    3-Lie algebra structure on $A\oplus A^*$ is given by
\begin{eqnarray*}
  &~~[e_2,e_3,e_4]=e_1,[e_1,e_3,e_4]=e_2,\\
  &~~[e^{\ast}_1,e^{\ast}_2,e^{\ast}_3]^{\ast}=ke^{\ast}_3+c_2e^{\ast}_4,[e^{\ast}_1,e^{\ast}_2,e^{\ast}_4]^{\ast}=c_1e^{\ast}_3-ke^{\ast}_4,\\
  &~~[e_2,e_3,e^{\ast}_1]=-e^{\ast}_4,[e_3,e_4,e^{\ast}_1]=-e^{\ast}_2,[e_2,e_4,e^{\ast}_1]=e^{\ast}_3,\\
  &~~[e_1,e_4,e^{\ast}_2]=e^{\ast}_3,[e_1,e_3,e^{\ast}_2]=-e^{\ast}_4,[e_3,e_4,e^{\ast}_2]=-e^{\ast}_2,\\
  &~~[e_3,e^{\ast}_1,e^{\ast}_2]=-ke_3-c_1e_4,[e_4,e^{\ast}_1,e^{\ast}_2]=-c_2e_3+ke_4,\\
  &~~[e_3,e^{\ast}_1,e^{\ast}_3]=ke_2,[e_4,e^{\ast}_1,e^{\ast}_3]=c_2e_2,\\
  &~~[e_3,e^{\ast}_1,e^{\ast}_4]=c_1e_2,[e_4,e^{\ast}_1,e^{\ast}_4]=-ke_2,\\
  &~~[e_3,e^{\ast}_2,e^{\ast}_3]=-ke_1,[e_4,e^{\ast}_2,e^{\ast}_3]=-c_2e_1,\\
  &~~[e_3,e^{\ast}_2,e^{\ast}_4]=-c_1ke_1,[e_4,e^{\ast}_2,e^{\ast}_4]=ke_1.
\end{eqnarray*}
They correspond to the double construction bialgebras $(A,\Delta)$ given in
Theorem~\ref{th:1111}, where $A$ is the 3-Lie algebra of Case~(4) given in Proposition~\ref{th:2.4}.

\item[(4)] The non-zero product of
    3-Lie algebra structure on $A\oplus A^*$ is given by
\begin{eqnarray*}
  &~~[e_1,e_2,e_4]=e_3,[e_1,e_e,e_4]=e_2,[e_2,e_3,e_4]=e_1,[e^{\ast}_1,e^{\ast}_2,e^{\ast}_3]^{\ast}=ce^{\ast}_4,\\
  &~~[e_1,e_3,e^{\ast}_2]=-e^{\ast}_4,[e_1,e_2,e^{\ast}_3]=-e^{\ast}_4,[e_1,e_4,e^{\ast}_2]=e^{\ast}_3,[e_1,e_4,e^{\ast}_3]=e^{\ast}_2,\\
  &~~[e_2,e_3,e^{\ast}_1]=-e^{\ast}_4,[e_2,e_4,e^{\ast}_1]=e^{\ast}_3,[e_2,e_4,e^{\ast}_3]=-e^{\ast}_1,[e_3,e_4,e^{\ast}_2]=-e^{\ast}_1,\\
  &~~[e_3,e_4,e^{\ast}_1]=-e^{\ast}_2,[e_4,e^{\ast}_1,e^{\ast}_2]=-ce_3,[e_4,e^{\ast}_1,e^{\ast}_3]=ce_2,[e_4,e^{\ast}_2,e^{\ast}_3]=-ce_1.
\end{eqnarray*}
They correspond to the double construction bialgebras $(A,\Delta)$ given in
Theorem~\ref{th:1111}, where $A$ is the 3-Lie algebra of Case~(7) given in Proposition~\ref{th:2.4}.
\end{enumerate}
\end{thm}
The proof is by a straightforward computation.

\smallskip

\noindent
{\bf Acknowledgements.}
This work  was supported by the Natural Science Foundation of China (Grant Nos.~11371178, 11425104).

\end{document}